\documentclass[11pt,a4paper,oneside,reqno]{amsart}
\usepackage{mathrsfs,amsmath,amsfonts,amssymb,amsthm,graphicx,color}

\usepackage{mathtools}
\usepackage{stmaryrd}
\numberwithin{equation}{section}
\newcommand{\p}{\partial}
 \renewcommand{\leq}{\leqslant}
\renewcommand{\geq}{\geqslant}
\renewcommand{\O}{\Omega}
\renewcommand{\div}{\operatorname{div}}
 \newtheorem{theorem}{{\bf Theorem}}[section]
\newtheorem{definition}[theorem]{Definition}

\newtheorem{lemma}[theorem]{{\bf Lemma}}
\newtheorem{prop}[theorem]{{\bf Proposition}}
\newtheorem{remark}[theorem]{{\bf Remark}}
\newcommand{\tr}{\operatorname{tr}}
\renewcommand{\a}{\alpha}
 \renewcommand{\b}{\beta}

\newcommand{\weight}[1]{\langle #1\rangle}

\newcommand{\sd}{\, \mathrm{d}}
\newcommand{\R}{\mathbb{R}}
\newcommand{\M}{\mathbb{R}}
\renewcommand{\S}{\mathbb{S}}

\renewcommand{\bar}{\overline}
\renewcommand{\tilde}{\widetilde}
\def\({\left(}
\def\){\right)}
\def\l|{\left|}
\def\r|{\right|}

\def\grad{\nabla}

\DeclareMathOperator{\diverg}{div}
\newcommand{\Lsym}[1]{D(#1)}  % symmetric part of velocity gradient L
\newcommand{\Lskew}[1]{W(#1)}
\def\I{\mathbb{I}_d}  % need to agree on a good symbol here
\def\calF{\mathcal{F}}
\newcommand{\N}{\ensuremath{\mathbb{N}}}

\newcommand{\barg}[1]{\bigl(#1\bigr)}
\newcommand{\Barg}[1]{\Bigl(#1\Bigr)}
\newcommand{\bset}[1]{\bigl\{#1\bigr\}}

\newcommand{\norm}[1]{\|#1\|}
\newcommand{\bnorm}[1]{\bigl\|#1\bigr\|}

\newcommand{\bbar}[1]{\bigl|#1\bigr|}

\newcommand{\bsqb}[1]{\bigl[#1\bigr]}

\newcommand{\scp}[2]{\langle #1, #2\rangle}
\newcommand{\bscp}[2]{\bigl< #1, #2\bigr>}

\newcommand{\seqn}[1]{(#1_n)_{n\in \N}}

\newcommand{\dv}[1]{\,{\rm d}#1}
\newcommand{\bref}[1]{(\ref{#1})}
\newcommand{\ts}[1]{\Omega_{#1}}  %% definition of time-space cylinder

\def\weakly{\rightharpoonup}
\def\weaklystar{ \stackrel{\ast}{\rightharpoonup}}

\def\BXzeroR{{B}_{X_0}(0,R)}
\def\BXzeroRone{{B}_{X_0}(0,R_1)}
\def\VWloka{\mathbb{V}}  %% YNL uses $V$ and $V(\Omega)$.
\def\HWloka{\mathbb{H}}  %% YNL uses $V$ and $V(\Omega)$.
\def\SonetwoWloka{H^{1}}
\def\StwotwoWloka{H^{2}}

\def\VNS{H^1_{0,\sigma}}
\def\HNSO{L^2_\sigma(\Omega)}
\def\VNSO{H^1_{0,\sigma}(\Omega)}
\def\VNSprime{H^{-1}_{\sigma}}
\def\VNSprimeO{H^{-1}_{\sigma}(\Omega)}
\def\HNSObb{(\HNSO)}

\def\Zspace{Z}

\def\tildeG{{G}}
\def\tildeL{\tilde{L}}
\def\uh{u_h}
\def\uhi{u_{hi}}
\def\uhone{u_{h1}}
\def\uhtwo{u_{h2}}
\def\Qh{Q_h}
\def\Qhi{Q_{hi}}
\def\Qhone{Q_{h1}}
\def\Qhtwo{Q_{h2}}

\def\CR{C_{\mathcal{R}}}     % why this funny index?

\newcommand{\dt}[1]{\partial_t #1} % use this for all t derivatives
\newcommand{\dtind}[2]{\partial_t #1_{#2}}  %% advanced technology for indices
\newcommand{\dtt}[1]{\partial_{tt}#1} % use this for all t derivatives

\begin{document}
\title[Beris--Edwards model]{Well-posedness of a fully-coupled Navier-Stokes/Q-tensor system with inhomogeneous boundary data}

\author{Helmut Abels, Georg Dolzmann, and YuNing Liu}
% \author{}
% \author{}
\address{Fakult\"at f\"ur Mathematik, Universit\"at Regensburg, 93040 Regensburg, Germany.}
\email{helmut.abels@mathematik.uni-regensburg.de}
%
% I think we put at most the email of the corresponding author.
%
 \email{georg.dolzmann@mathematik.uni-regensburg.de}
 \email{yuning.liu@mathematik.uni-regensburg.de}



\thanks{The third author gratefully acknowledges partial financial support
by the DFG through grant AB285/4-2}

\begin{abstract}
We prove short-time well-posedness and existence of global weak solutions of the Beris--Edwards model for 
nematic liquid crystals in the case of a bounded domain with inhomogeneous mixed Dirichlet and Neumann boundary conditions. The system consists of the Navier-Stokes equations coupled with an evolution 
equation for the $Q$-tensor. The solutions possess higher regularity in time of order one compared to the class of weak solutions with finite energy. This regularity is enough to obtain Lipschitz continuity of the non-linear terms in the corresponding function spaces. Therefore the well-posedness is shown with the aid of the contraction mapping principle using that the linearized system is an isomorphism between the associated function spaces.

\end{abstract}

\maketitle

%%%%%%%%%%%%%%%%%%%%%%%%%%%%%%%%%%%%%%%%%%%%%%%%%%%%%%%%%%%%%%
%%                                                          %%
%% I N T R O D U C T I O N                                  %%
%%                                                          %%
%%%%%%%%%%%%%%%%%%%%%%%%%%%%%%%%%%%%%%%%%%%%%%%%%%%%%%%%%%%%%%

\section{Introduction}

We study the well-posedness of a model for the instationary flow of a nematic liquid crystal described by a model due to Beris and Edwards, cf.~\cite{BerisEdwards1994}. In this model the orientation and degree of ordering of the liquid crystal is described by a symmetric, traceless $d\times d$ tensor $Q$. This description goes back to Landau and DeGennes, cf.~\cite{DeGennesProst1995}. In the case that the tensor is uniaxial, i.e.,
it has two equal non-zero eigenvalues, it can be represented as
\begin{equation*}
  Q=s\left(n\otimes n-\frac 1d\,\I\right)\,,
\end{equation*}
where the scalar order parameter $s\in [-\frac 12,1]$  measures the
degree of orientational ordering and $n$ is a unit vector and describes the direction of orientation. The Beris-Edwards models leads to a system, which couples the incompressible Navier-Stokes equations with a second order parabolic equation for the evolution of the tensor $Q$. More precisely, we consider
\begin{align}\label{strongformeqn}
\begin{split}
 \dt{u} + (u\cdot \grad)u  + \grad p
&\, = \diverg\barg{\nu(Q)\Lsym{u} }+ \diverg\barg{\sigma(Q,H)+ \tau(Q,H)}\,,\\
\diverg u &\, = 0\,, \\
 \dt{Q} + \barg{  u \cdot \grad } Q - S(\grad u, Q) &\, = \Gamma H(Q)
 \end{split}
\end{align}
in $\ts{T}=\Omega\times (0,T)$ for a sufficiently smooth bounded domain $\Omega\subseteq\R^d$, $d=2,3$, $T>0$. Here  $\sigma$ is a skew-symmetric tensor and $H$, $\tau$, and $S$ are symmetric tensors given by
\begin{align}\label{tensors}
\begin{split}
  H &=\lambda \Delta Q - aQ + b\barg{Q^2 - \tfrac{1}{d}\,\tr(Q^2)\,\I} - c\tr(Q^2) Q,\\
 \sigma(Q,H) &\, = QH - HQ = Q \Delta Q - \Delta Q Q\,,\\
\tau(Q,H) &\,= -\lambda \grad Q\odot \grad Q
-\xi  \barg{Q + \tfrac{1}{d}\,\I} H - \xi H  \barg{Q + \tfrac{1}{d}\,\I}
+ 2\xi  \barg{Q + \tfrac{1}{d}\,\I}\tr(QH)\,,\\
 S(\nabla u, Q) &\, %= S(\Lskew{u} , Q)
= \barg{\xi \Lsym{u}+\Lskew{u}}  \barg{Q + \tfrac{1}{d}\,\I}
 + \barg{Q + \tfrac{1}{d}\,\I} \barg{\xi \Lsym{u}-\Lskew{u}}  \\
&\,\qquad {}- 2\xi \barg{Q + \tfrac{1}{d}\,\I}\tr(Q\grad u)
 \,,
\end{split}
\end{align}
where we used the notation
\begin{align*}%\label{Ldecomposition}
\Lsym{u}  = \frac{1}{2}\barg{ \nabla u + (\nabla u)^T }\,,\quad
 \Lskew{u} = \frac{1}{2}\barg{\nabla u - (\nabla u)^T}
\end{align*}
for the stretch and the vorticity tensor, respectively. Moreover, $\Gamma$, $\lambda$, $a$, $b$, and $c$ are positive constants. 
We note that $S(\nabla u,Q)$ is introduced to describe how the flow gradient rotates and stretches the director field.

 Here $H$
relates to the variational derivative of the free energy functional
% with respec to the constraint $\tr Q=0$
 which uses
the one-constant approximation for the Oseen-Frank energy of liquid crystals
together with a Landau-DeGennes expression for the bulk energy
\begin{align}\label{freeenergy}
 \calF(Q) = \int_\Omega \barg{\frac{\lambda}{2}\, |\nabla Q|^2 + f_B(Q)}\dv{x}\, ,
\end{align}
where the bulk energy $f_B$ is given by
\begin{align*}
 f_B(Q) = \frac{a}{2}\,\tr(Q^2) -\frac{b}{3}\,\tr(Q^3)
+\frac{c}{4}\,\tr(Q^4)  \,.
\end{align*}
Hence $H=H(Q)$ can be rewritten as
\begin{align}\label{strongforcing}
 H(Q) &\, = \lambda \Delta Q + L,\quad L = -aQ + b\barg{Q^2 - \frac{1}{d}\,\tr(Q^2)\,\I} - c\tr(Q^2) Q,
\end{align}
where $L=-Df_B(Q)$ consists of lower-order terms in the equation.

We complement this system \eqref{strongformeqn}-\eqref{tensors} by the initial condition
\begin{align}\label{initialconditions}
\begin{split}
(u,Q)|_{t=0}&\,=(u_0,Q_0)%\in \HNSO
\quad\text{in}\ \Omega
\end{split}
\end{align}
 and the Dirichlet-Neumann boundary conditions of mixed type,
\begin{align}\label{mixedboundaryconditions}
\begin{alignedat}{2}
u&\, =0 &&\text{ on }(0,T)\times\p\O\,,\\
Q\,&=Q_D         &&\text{ on }(0,T)\times\Gamma_D\,,\\
\p_n Q\,&=Q_N  \quad       &&\text{ on }(0,T)\times\Gamma_N\,,
\end{alignedat}
\end{align}
where $\p\O=\Gamma_D\cup\Gamma_N$ and
$\Gamma_D,\Gamma_N$ are closed, disjoint subsets of $\mathbb{R}^d$ and $(Q_D,Q_N)$ will be independent of $t\in (0,T)$ in the following.% such that $\Gamma_D\cap \Gamma_N=\emptyset$ has $(d-1)$-dimensional Hausdorff measure zero.

So far there are only a few results on the mathematical analysis of this system. First contributions were given by Paicu and Zarnescu. In \cite{PaicuZarnescuARMA2012} the authors consider the case $\xi=0$, $\Omega=\R^d$. They prove existence of weak solutions for $d=2,3$ as well as higher regularity of weak solutions and weak-strong uniqueness if $d=2$. In  \cite{PaicuZarnescuSIMA2011} existence of weak solutions is proved provided that $\xi$ is sufficiently close to $0$ and $\Omega=\R^d$, $d=2,3$. Wilkinson studied in \cite{Wilkinson} the system \eqref{strongformeqn}-\eqref{tensors} under periodic boundary condition in the case that $f_B$ is replaced by a certain singular potential. The potential guarantees that $Q$ attains only physically reasonable values. He established existence of weak solutions for a general $\xi$ and higher regularity in the case of two space dimensions and $\xi=0$. Finally, Feireisl et al.~\cite{FeireislEtAlQTensor} derived a non-isothermal variant of the Beris-Edwards system and proved existence of weak solutions for this system in the case of a singular potential and for periodic boundary conditions. Recently, Wang et al.  establish in \cite{wangzz} a rigorous derivation from Beris-Edwards system to the Ericksen-Leslie system, which is widely investigated in the literature. Here we  refer to recent works \cite{huanglinwang2014}, \cite{linlinwang2010}, \cite{wangzzericksen}, \cite{wuxuliu2013} and the references therein for more details.

In the present paper we discuss existence of weak solutions in a bounded domain with mixed Dirichlet-Neumann boundary conditions as well as well-posedness of the system in a class of solutions, which possess higher regularity in time than the class of weak solutions. These solutions are not necessarily more regular with respect to the space variable. We note that in the case without boundary, one could  establish higher regularity in space for these solutions by using e.g. standard difference quotient techniques. But in the present case with boundary conditions we do not have an appropriate regularity result for the principal part of the linearized system, which is a Stokes system coupled with an elliptic equation for $Q$ through the terms $S(\nabla u, \tilde{Q})$ and $\diverg \sigma(\tilde{Q},H)$ for a suitable $\tilde{Q}$. These coupling terms cancel in the standard energy argument. However, they give rise to extra boundary integrals, when testing with higher order spacial derivatives of the solution, which cannot be absorbed. Fortunately, for higher order temporal derivatives these boundary terms vanish again.
 The main novelty in the paper is to use this observation together with the fact that one more temporal derivative (compared to the regularity class of weak solutions) is enough to prove Lipschitz continuity of the non-linear terms in the associated function spaces. Therefore we are able to prove existence of unique solutions in this regularity class for sufficiently short times. Let us note that we expect that our solutions also possess the natural higher regularity with respect to the spacial variables. This might be a future work. But to obtain well-posedness of the system locally in time such a regularity result is not needed. %%% Remark: \nu=\nu(Q)

In order to formulate our main results, we have to introduce some assumptions and notation.
In the sequel, we shall assume that $\Gamma=\lambda=a=b=c=1$ to simplify the notation. But all results hold true for general values of these constants if $c>0$.
In the following we assume that $\Omega$ is a bounded domain with $C^3$-boundary and
\begin{equation}\label{viscosity}
  \nu\in C^2(\mathbb{R}^{d\times d}),\quad 0<c_0\leq \nu(\cdot)\leq c_1<\infty
\end{equation}
for some constants $c_0,c_1$. In the following $\S_0$ denotes the vector space of all symmetric and trace free $d\times d$ matrices. More details on the notation are given in Section~\ref{sec:Notation} below.
% and we assume that $Q_D=\widetilde{Q}_D|_{\Gamma_1}$ for some $\widetilde{Q}_D\in H^{\frac32}(\partial\Omega) $.
We use the following notion of weak solution.

\begin{definition}\label{defweaksolution}
Suppose that $T>0$, $u_0\in L^2_\sigma(\Omega)$, $Q_0\in H^1(\Omega;\mathbb{S}_0)$, $Q_D\in H^{\frac32}(\Gamma_D;\mathbb{S}_0)$, and $Q_N\in H^{\frac12}(\Gamma_N;\mathbb{S}_0)$.  A pair $(u,Q)$ with
\begin{align*}
u&\, \in BC_w([0,T]; \HNSO)\cap L^2(0,T;\VNSO)\,,\\
Q&\,\in BC_w([0,T];H^1(\O;\S_0))\cap L^2(0,T;H^2(\O;\S_0))%,~\Delta Q\in L^2(\O_T;\S_0)
\end{align*}
is called a weak solution of the system \eqref{strongformeqn} in $\ts{T}$
with initial conditions~\bref{initialconditions} and boundary conditions
\eqref{mixedboundaryconditions} if the following holds:
 \begin{enumerate}
 \item For any $v\in C^1([0,T];H^1_{0,\sigma}(\Omega)\cap W^{1,\infty}(\O;\R^d))$ and $\Psi\in C^1([0,T];H^1(\O;\S_0))$ with $v|_{t=T}=\Psi|_{t=T}=0$,
     it holds that
     \begin{equation}\label{weakforufield}
\begin{split}
  &\int_{\ts{T}}\big(-u\cdot \dt{v}+(u\cdot\nabla u)\cdot v
+\nu(Q)  \Lsym{u}:\Lsym{v}\big)\, \dv(x,t)\\
&+\int_{\ts{T}}\left(\barg{ \sigma+\tau }(Q,H(Q))
 \right):\nabla v \, \dv(x,t)
=\int_\O u_0 v|_{t=0} \, {\rm d}x
\end{split}
\end{equation}
and
\begin{align}\nonumber
  -\int_{\Omega_T}Q:\dt{\Psi}\, \dv(x,t)&+\int_{\Omega_T} u\cdot\nabla Q:\Psi\, \dv(x,t)-\int_{\Omega_T} S(\nabla u,Q):\Psi\, \dv(x,t)\\\label{weakforq}
&=\int_{\Omega_T}  H(Q):\Psi\, \dv(x,t)+\int_\Omega Q_0:\Psi|_{t=0} \, {\rm d}x.
\end{align}
   \item  For almost every $t\in (0,T)$ the following energy inequality holds:
\begin{align*}%\label{additionalenergy}
\begin{split}
&\,  \frac 12\int_{\O}|u(t,x)|^2\, {\rm d}x+\mathcal{F}\left(Q(t,\cdot)\right)+\int_{\Omega_t}
 \left( \nu(Q(\tau,x))|Du(\tau,x)|^2+  |H(Q(\tau,x))|^2\right)\,{\rm d}(x,\tau)\\
 &\,\qquad \leq \frac 12\int_{\O}|u_0(x)|^2\,{\rm d}x+\mathcal{F}\left(Q_0\right).
\end{split}
\end{align*}
\item For almost every $t\in [0,T]$, $Q|_{\Gamma_D}=Q_D$ and $\frac{\partial Q}{\partial n}|_{\Gamma_N}= Q_N$.
%    \begin{equation}\label{weakbc}
 %     \int_{\O}\(\Delta Q:\Psi+\nabla Q:\nabla\Psi\)=0,~\forall \Psi\in H^1(\O;\S_0)~\text{with}~\Psi|_{\Gamma_D}=0.
%    \end{equation}
 \end{enumerate}
\end{definition}
Throughout the paper $\Omega$ is a bounded domain with $C^3$-boundary.
Our first result is a result on global existence of weak solutions in the case of homogeneous Neumann boundary
conditions for the director field $Q$.

\begin{theorem}[Existence of weak solutions]\label{weaksolution}
 Let $T, Q_D, Q_N, u_0$ be as in Definition~\ref{defweaksolution}. Then the system \eqref{strongformeqn} has a global weak solution for any $T>0$.
\end{theorem}

% Note carefully that we do not assume that $\Gamma_D\cap\Gamma_N=\varnothing$, a situation which
% may lead to a loss of global $H^2$-regularity for $Q$ in space. If we make the additional assumption
% that $\Gamma_D\cap\Gamma_N=\varnothing$, then we obtain $H^2$ regularity globally
% in space,  $Q\in L^2(0,T;H^2(\O;\S_0))$, which is stronger than the regularity
% required in the definition of a weak solution in Definition~\ref{defweaksolution}. In this case, \eqref{weakbc} can be improved by $\p_n Q|_{\Gamma_N}=0$. Actually, if $Q(t)\in H^2(\O)$ for some $t\in [0,T]$,  we can show, by making use of \eqref{weakbc} and $\Psi|_{\Gamma_D}=0$, that $\int_{\Gamma_N}\p_n Q:\Psi=0$ and the assertion follows by a density argument.

Our second result concerns regularity in time for weak solutions of the system~\bref{strongformeqn}.
This result requires a subtle compatibility condition related to the initial data
for $Q$. As \eqref{strongformeqn} is an evolution equation, it can be written in the
abstract form
 \begin{equation}\label{stanon}
  \frac {d}{dt} (u,Q) =\mathcal{E}(u,Q)
 \end{equation}
where $\mathcal{E}:\VNSO\times H^2(\O)\to \VNSprimeO\times L^2(\O)$ is defined by
 \begin{equation}
   \begin{split}
\bscp{\mathcal{E}(u,Q)}{(\varphi,\Psi)}
=&-\int_\Omega (-u\otimes u +\nu(Q)\Lsym{u}+(\tau+\sigma)(Q,H)):\nabla\varphi\, {\rm d}x
 \\
&+ \int_\Omega(( u \cdot \nabla )Q+S+\Gamma H): \Psi\, {\rm d}x
   \end{split}
 \end{equation}
for all $(\varphi,\Psi)\in \VNSO\times H^2(\O;\S_0)$.
Since \eqref{mixedboundaryconditions} specifies a time-independent Dirichlet boundary condition,
it follows that $\dt{Q}|_{\Gamma_D}=0$ and this observation leads to the compatibility
condition that the trace of the second component
on the right-hand side of \eqref{stanon} vanishes on $\Gamma_D$.
% If the solution is sufficiently
% regular in time, then the same condition has to hold for the initial data for $Q$.
Consequently we define the
phase space,
\begin{equation*}%\label{init}
\Zspace =\bset{  (u,Q)\in \VNSO\times H^2(\O): \mathcal{E}(u,Q)
\in L^2_{\sigma}\times H^1_{\Gamma_D},\,
Q=Q_D\text{ on }\Gamma_D,\,
\p_n Q\,=Q_N \text{ on }\Gamma_N
% ~\text{satisfies} ~\eqref{mixedboundaryconditions}
},
\end{equation*}
where
\begin{align*}
H^1_{\Gamma_D}=H^1_{\Gamma_D}(\O;\S_0):=\{Q\in H^1(\O;\S_0): Q|_{\Gamma_D}=0\}\,.
\end{align*}
Note that the phase space defined above is non-empty. For instance, if we choose $u_0\in H^2(\O)\cap H^1_{0,\sigma}(\O)$ and $Q_0\in H^3(\O)$ satisfying the boundary condition in \eqref{mixedboundaryconditions} such that $\Delta Q_0|_{\Gamma_D}=0$, then $(u_0,Q_0)\in Z$.

\begin{theorem}[Local existence and  uniqueness of solutions with regularity in time] \label{localstrong}~\\
Suppose that %$\Gamma_D\cap\Gamma_N=\varnothing$, that the boundary data satisfy
  \begin{equation*}%\label{boundaryvaluemixed}
     (Q_D,Q_N)\in H^{\frac 52}(\Gamma_D;\S_0)\times H^{\frac 32}(\Gamma_N;\S_0)
  \end{equation*}
and that the initial data satisfy $(u_0,Q_0)\in \Zspace$. Then there exists some $T>0$ such that
 the system \eqref{strongformeqn} together with \eqref{initialconditions} and \eqref{mixedboundaryconditions} has a unique solution $(u,Q)$ with
\begin{align*}
 u&\, \in H^2(0,T;\VNSprimeO)\cap H^1(0,T;\VNSO)\,,\\
Q&\, \in H^2(0,T;L^2(\O;\S_0))\cap H^1(0,T;H^2(\O;\S_0))\,.
\end{align*}
\end{theorem}

\begin{remark}
  Theorems \ref{weaksolution} and \ref{localstrong}
are also valid for $\Gamma_D = \emptyset$ or $\Gamma_N = \emptyset$.
\end{remark}

Finally, we note that the idea to use higher regularity in time to obtain a unique solution of \eqref{strongformeqn}-\eqref{mixedboundaryconditions} (in the case $\xi=0$) has also been used by Guill\'{e}n-Gonz\'alez and Rodr\'iguez-Bellido~\cite{GuillenGonzalez}.

%%%%%%%%%%%%%%%%%%%%%%%%%%%%%%%%%%%%%%%%%%%%%%%%%%%%%%%%%%%%%%
%%                                                          %%
%% P R E L I M I N A R I E S                                %%
%%                                                          %%
%%%%%%%%%%%%%%%%%%%%%%%%%%%%%%%%%%%%%%%%%%%%%%%%%%%%%%%%%%%%%%

\section{Preliminaries}

\subsection{Notation}\label{sec:Notation}
For two vectors $a$, $b\in\R^d$ we set $a\cdot b = \sum_{i=1}^d a_i b_i$ and $a\otimes b=ab^T=(a_ib_j)_{1\leq i,j\leq d}$
and for two matrices $A$, $B\in\R^d$ we set $A: B
= \sum_{i,j=1}^d A_{ij} B_{ij}= \operatorname{tr}(A^TB)$. Then
\begin{align}\label{linalg}
 (AB):C=B:(A^TC)=A:(CB^T)\quad
\text{ for all }A,\, B,\, C\in \R^{d\times d}
\end{align}
and we omit the parentheses for simplicity in the following
if they are clear from the context. % that the equation is scalar.
Einstein's summation convention
is applied throughout the paper if repeated indices are written.
We define the norm of a matrix $A\in \M^{d\times d}$
by $|A|^2=\tr(A^TA)=A:A$. For a differentiable matrix-valued function $F\colon \Omega\to \R^{d\times d}$, we denote by $\div F=(\p_{\b}F_{\a\b})_{1\leq \alpha \leq d}$
the vector whose $\alpha$th component is the divergence of the $\alpha$th row in $F$.
Moreover, if $A,B\colon \Omega \to \R^{d\times d}$ are differentiable, we introduce the contraction $\odot$ by
$$
\nabla A\odot\nabla B
=\barg{ \p_iA_{\a\b}\p_jB_{\a\b} }_{1\leq i,j\leq d}
=\barg{ \p_iA:\p_jB }_{1\leq i,j\leq d}\,
$$
and $Q:\nabla A= (Q:\partial_j A)_{1\leq j\leq d}$. Finally, $\mathbb{I}_d$ denotes the matrix, which represents the identity on $\R^d$.

For the following it is convenient to rewrite the definitions of $\tau$ and $S$ as
\begin{align*}
\tau(Q,H) &\, = \tau_1(Q)+\xi\tau_2(Q,H)-\frac {2\xi}{d}\, H, \\
S(\nabla u,Q)&\, = S_1(\nabla u,Q)+\xi S_2(\nabla u,Q)+\frac{2\xi}{d}\,\Lsym{u}\,,
\end{align*}
with
\begin{align}\label{tau1}
\begin{aligned}
\tau_1(Q)&\, =-\nabla Q\odot \nabla Q-\frac 1d\, \I\,\tr Q^2 \,,\\
\tau_2(Q,H)&\,  =-QH-  HQ+2 (Q+\frac 1d\, \I)\tr(QH)\,,\\
S_1(\nabla u,Q)&\, = \Lskew{u} Q-Q\Lskew{u}\,,\\
S_2(\nabla u,Q)&\, =\Lsym{u}Q+Q\Lsym{u}-2(Q+\frac 1d\, \I)\tr (Q\nabla u)\,.
\end{aligned}
\end{align}
We note that
\begin{alignat*}{1}
  -\diverg \tau_1(Q)&= \Delta Q \odot \nabla Q + \nabla \left(\frac1d \operatorname{tr} Q^2 +\sum_{j=1}^n \frac12|\partial_j Q|^2 \right)\\
 &= H(Q):\nabla Q + \nabla   \left(\frac1d \operatorname{tr} Q^2 +\sum_{j=1}^n \frac12|\partial_j Q|^2 +f_B(Q)\right).
\end{alignat*}
Therefore
\begin{equation}\label{eq:tau1}
  \int_{\ts{T}} \tau_1(Q):\nabla v\dv{(x,t)} = \int_{\ts{T}} (H(Q):\nabla Q)\cdot v\dv{(x,t)}
\end{equation}
for all $Q$ and $v$ as in Definition~\ref{defweaksolution}.
% Flow equations with additional forcing term ($\grad^3 Q)$, coupled
% to a parabolic equation in $Q$.

Finally, $\langle x',x\rangle_{X',X}$ denotes the duality product of $x'\in X'$ and $x\in X$ for a Banach space $X$ and $(.,.)_H$ denotes the inner product of a Hilbert space $H$.

\subsection{Function spaces}
We use standard notation for the Lebesgue and Sobolev spaces $L^p(\Omega)$ and
$W^{k,p}(\Omega)$ as well as $L^p(\Omega;M)$ and $W^{k,p}(\Omega;M)$ for the corresponding spaces for $M$-valued functions.
Sometimes we omit the domain and the range if they are clear from the context.
The $L^2$-based Sobolev spaces are denoted by $H^k(\Omega)$ and $H^k(\Omega;M)$. The usual spaces of
divergence free vector fields are introduced by 
\begin{align*}
 \HNSO &\, %=  \HNS(\Omega;\R^d)
=\bset{u\in L^{2}(\O;\R^d),\,\diverg u = 0,\,\gamma(u)=0}\,,\\
 \VNSO &\,% =  \VNS(\Omega;\R^d)
=\bset{u\in H^{1}_0(\O;\R^d),\,\diverg u = 0}\,,\quad
H^{-1}_\sigma(\Omega) = (\VNSO)'
\end{align*}
where $\gamma(u)=u\cdot n\in H^{-\frac12}(\partial\Omega)$ is defined a generalized trace sense, cf. e.g.~\cite{BuchSohr}. % in the sense of trace and where $\nu$
%is the exterior normal to $\partial\Omega$.
% Other common notations for these spaces are $\HNS(\Omega)=\Ltwosigma$ and
% $\VNSO=\Honesigma$.
Note that
\begin{align*}
 L^2(\Omega;\R^d) = \HNSO \oplus {\HNSObb}^\perp \,\text{ with }\,
{\HNSObb}^\perp = \bset{ u \in L^2(\Omega;\R^d),\,u=\nabla q
\text{ for some }q\in H^1(\Omega) }\,.
\end{align*}
The Helmholtz projection, i.e., the orthogonal projection $L^2(\Omega;\R^d) \to \HNSO$,
is denoted by $P_\sigma$. We refer to \cite{BuchSohr} for its basic properties. For $f\in H^{-1}(\Omega;\R^d)$ we define $P_\sigma f \in H^{-1}_{\sigma}(\Omega)$ by $P_\sigma f=f|_{H^1_{0,\sigma}(\Omega)}$. Moreover, for $F\in L^2(\Omega;\R^{d\times d})$ we define $\diverg F\in H^{-1}(\Omega;\R^d)$ by 
\begin{equation*}
  \weight{\diverg F, \Phi}_{H^{-1},H^1_0}= -\int_{\Omega} F: \nabla \Phi \dv{x}\qquad \text{for all }\Phi\in H^1_0(\Omega;\R^d)\, .
\end{equation*}
%  Note that $P_\sigma$ maps $H^1_0(\Omega;\R^d)$ into
% $H^1(\Omega;\R^d)\cap L^2_\sigma(\Omega)$ continuously % \textbf{there is no zero from what I know, is the duality okay?
% %State estimates for the operator norm.}
% and we extend $P_\sigma$ by duality to a map
% % $P_{\sigma}:H^{-1}(\O)^d\mapsto \VNS'(\O)$ by
% \begin{align*}
% P_{\sigma}:H^{-1}(\O;\R^d)\to \VNSprimeO,\quad
%  \langle P_{\sigma}f\,, v\rangle=\langle f, P_{\sigma}v\rangle\,
% \quad\text{ for all } v\in H^1_0(\O;\R^d)\,.
% \end{align*}

%\textbf{define here or elsewhere the good properties for $P_\sigma$ along with a reference}
Finally, if $X$ is a Banach space and $T>0$, then $C([0,T];X)$ and $BC_w([0,T];X)$ denotes the space of all bounded $f\colon [0,T]\to X$ that are continuous or weakly continuous, respectively.

\subsection{An algebraic identity}
The following cancellation property plays an important role in the sequel.

\begin{lemma}
Let  $Q_1$, $Q_2\in L^2(\O;\S_0)$ and $u\in \VNSO$. Then
\begin{equation}\label{tensor4}
    S(\nabla u,Q_1): Q_2
+\left( \sigma(Q_1,Q_2)+\xi \tau_2(Q_1,Q_2)-\frac{2\xi}{d}Q_2 \right):\nabla u=0\,.
  \end{equation}
 \end{lemma}

  \begin{proof}
  We are going to prove the following two identities
  \begin{equation*}
    \sigma(Q_1,Q_2):\nabla u+S_1(\nabla u,Q_1): Q_2=0,\quad
\tau_2(Q_1,Q_2):\nabla u+S_2(\nabla u,Q_1): Q_2=0\,.
  \end{equation*}
These identities together with
$-\frac{2\xi}{d}Q_2 :\nabla u+\frac{2\xi}{d}\Lsym{u}: Q_2=0$ imply the assertion.
We use~\bref{linalg} and the symmetries to obtain
\begin{align*}
2 S_1(\nabla u, Q_1):Q_2
&\, = \nabla u Q_1: Q_2
-(\nabla u)^T Q_1: Q_2
-Q_1  \nabla u:  Q_2
+Q_1 (\nabla u)^T  :Q_2 \\
&\, = 2\barg{\nabla u :Q_2 Q_1 -  \nabla u : Q_1Q_2   }
= -2 \nabla u : \sigma(Q_1, Q_2)\,.
\end{align*}
Similarly, using the symmetry of $Q_1$ and $Q_2$,
\begin{equation*}
 \begin{split}
    &\tau_2(Q_1,Q_2):\nabla u+S_2(\nabla u,Q_1):Q_2\\
&\,\quad  = \barg{ -Q_1Q_2-  Q_2Q_1+2 (Q_1+\frac 1d\, \I)(Q_1:Q_2) }:\nabla u \\
&\,\quad\qquad
+ \barg{ \Lsym{u}Q_1+Q_1\Lsym{u}-2(Q_1+\frac 1d\, \I) (Q_1:\nabla u) } : Q_2 =0\,,
 \end{split}
\end{equation*}
where we use that $Q_2$ and $\nabla u$ are trace free. The proof is complete.
 \end{proof}

\subsection{Orthogonal bases of eigenfunctions}

The idea is to solve the system \eqref{strongformeqn} by a Galerkin approach
based on eigenfunctions of the Laplace operator for $Q$ and eigenfunctions of
the Stokes operator for $u$. The existence of these bases follows from the
spectral theorem for compact operators in Hilbert spaces and standard results for elliptic equations and the stationary Stokes system. We summarize their
properties. Note that we consider the Laplace equation as a vector-valued equation
with values in $\S_0$.

%\textbf{if the smallest eigenvalue is positive then we don't have a pure Neumann problem}

\begin{lemma}[Eigenvalue problem of the Laplace operator]\label{eigenlaplace}~\\
   There exists an orthonormal basis $\seqn{e}\subset H^2(\O;\S_0)$
of $L^2(\O;\S_0)$ and a non-decreasing sequence of corresponding eigenvalues
 $\seqn{\lambda} \subset \mathbb{R}^+$
with $\lim_{n\to\infty}\lambda_n=\infty$ such that
\begin{align*}
 \begin{alignedat}{2}
 -\Delta e_n&=\lambda_n e_n \quad && \text{ in }\Omega\,, \\
e_n&=0 && \text{ on }\Gamma_D\,,\\
\frac{\p e_n}{\p n}&=0&& \text{ on }\Gamma_N\,.
 \end{alignedat}
\end{align*}
   \end{lemma}

A similar result holds for Stokes operator $A:=-P_{\sigma}\Delta$.

\begin{lemma}[Eigenvalue problem of the Stokes operator]\label{eigenstokes}~\\
  There exists an orthonormal basis
$\seqn{v}\subset \VNSO\cap W^{1,\infty}(\O;\R^d)$
of $L^2_{\sigma}(\O;\R^d)$ of eigenfunctions and a non-decreasing sequence
$\seqn{\omega}\subset \mathbb{R}^+$ of corresponding eigenvalues
with $\lim_{n\to\infty}\omega_n=\infty$ and $A v_i=\omega_i v_i$ for all $i\geq 1$.
   \end{lemma}

The regularity result $v_n\in W^{1,\infty}(\O;\R^d)$
follows from the standard regularity theory for the Stokes system
provided that  $\p\O\in C^3$, cf. e.g.~\cite{Galdi}.

This result  allows us to define the fractional Stokes operator
 $A^{\frac m2}: \mathcal{D}(A^{\frac m2})\to L^2_{\sigma}(\O;\R^d)$
for $m\in\mathbb{N}$, $m\geq 1$. The following lemma gives the regularity
of functions in $\mathcal{D}(A^{\frac m2})$ and the proof can
be found in \cite[Proposition 4.12]{ConstantinFoias1988}.

\begin{lemma}\label{fractional}
  Let $\O\subset\R^d$ be an open and bounded set of class $C^{\ell}$ with $\ell\geq 1$.
Then $\mathcal{D}(A^{\frac m2})$ is contained in %\subset
$H^m(\O;\R^d)\cap \VNSO$ provided $1\leq m\leq \ell$.
\end{lemma}

We recall a compactness result of Aubin-Lions type,
see~\cite{Simon1987} for the proof.

\begin{lemma}\label{compact}
Suppose that $p_1,\,p_2\in(1,\infty)$.
Assume that $X_{-1}$, $X_0$ and $X_1$ are three separable and reflexive Banach spaces
such that the inclusion $X_1\hookrightarrow X_0$ is compact and the inclusion
$X_0\hookrightarrow X_{-1}$ is continuous.
If $(u^{(m)})_{m\in \N}$ is a bounded sequence in $L^{p_1}(0,T;X_1)$
such that  $(\dt{u^{(m)}})_{m\in \N}$ is bounded in $L^{p_2}(0,T;X_{-1})$,
then there exists a subsequence $u^{(m')}$ which converges in
$L^{p_1}(0,T;X_0)$.
Furthermore, if $X_0=(X_{-1},X_1)_{[1/2]}$, where $(.,.)_{[\theta]}$ denotes the complex interpolation functor, and $p_1=\infty$,
then  there exists a subsequence $u^{(m')}$ which converges in $C([0,T];X_0)$.
\end{lemma}

The next interpolation result is stated in the three-dimensional situation which
is the main focus of this paper.

%\textbf{check wether this Gelfand needs compactness and check application,
%stress that $C$ is independent of time, if it is.}

\begin{lemma}[Interpolation]
%Let $\O\subset\R^d$, $d\leq 3$, be an open and bounded domain with boundary of class $C^3$. Then
There is some $C>0$ such that for all $f\in H^2(\Omega)$ the estimates
\begin{equation}\label{interpo1}
 \|f\|_{L^{\infty}(\O)}\leq C\|f\|_{H^1(\O)}^{\frac 12}\|f\|_{H^2(\O)}^{\frac 12}\,,
\end{equation}
and
\begin{equation}\label{interpo3}
  \|f\|_{L^{\infty}(\O)}\leq C\|f\|_{L^2(\Omega)}^{\frac 14}\|f\|_{H^2(\Omega)}^{\frac 34}
\end{equation}
hold.
If, additionally, $H_1\hookrightarrow H \hookrightarrow H_1'$ is a Gelfand-triple, then
  \begin{equation}\label{interpo2}
  \|f\|^2_{C([0,T];H)}\leq 2(\|f\|_{H^1(0,T;H_1')}\|f\|_{L^2(0,T;H_1)}+\|f(0,\cdot)\|^2_{H}).
  \end{equation}
%where $C>0$ is independent of $T>0$.
\end{lemma}
Proofs of these statements can be found in \cite[Section 2.1]{Abels2009}.

\begin{lemma}\label{temam}
  Let $X$ and $Y$ be two Banach spaces such that $X\subset Y$ with a continuous injection.
If $f\in L^{\infty}(0,T;X)$ is weakly continuous with values in $Y$,
then $f$ is weakly continuous with values in $X$.
\end{lemma}

The proof can be found in \cite[pp.263]{Temam1984}.

%%%%%%%%%%%%%%%%%%%%%%%%%%%%%%%%%%%%%%%%%%%%%%%%%%%%%%%%%%%%%%
%%                                                          %%
%% G L O B A L  W E A K  S O L U T I O N S                  %%
%%                                                          %%
%%%%%%%%%%%%%%%%%%%%%%%%%%%%%%%%%%%%%%%%%%%%%%%%%%%%%%%%%%%%%%

\section{Existence of weak solutions and proof of Theorem~\ref{weaksolution}}
%\textbf{uniqueness??? (Yuning: finite dimension system has unique solution by Picard and the limit is not unique) HA: not needed}

This section is devoted to the proof of the existence of global weak solutions
via a modified Galerkin method introduced in \cite{LinLiu1995}.
In view of Lemma \ref{eigenlaplace} and Lemma \ref{eigenstokes}, we
define the finite-dimensional Banach spaces
\begin{align*}
E_n&\, =\operatorname{Span}\{e_1,\ldots,e_n\}\subset H^2(\Omega;\S_0)
\subset L^2(\O;\mathbb{S}_0) \,,\\
V_n&\, =\operatorname{Span}\{v_1,\ldots,v_n\}\subset  \VNSO\cap W^{1,\infty}(\O;\R^d)
\subset \HNSO\,,
\end{align*}
along with  the orthogonal projection operators
\begin{equation*} %\label{pi} %\label{p1} was destroyed
  \pi_n  : L^2(\O;\mathbb{S}_0)\longrightarrow E_n\quad \text{ and }\quad
  \mathcal{P}_n : L^2_{\sigma}(\O)\longrightarrow V_n\,.
\end{equation*}
%\textbf{check these domains and the boundary conditions}
These two orthogonal projection operators are bounded
linear operators with norms bounded by one, a fact which  will be
used in the calculations in this section without being mentioned. To simplify notation we use
$(\cdot,\cdot)_{\O}$ to denote the inner product in $L^2(\O;\R^N)$, $N\geq 1$, and
in $L^2_{\sigma}(\O)$.
% For $A,B\in L^2(\O;\R^{d\times d})$, we denote $(A,B):=\int_{\O}A(x):B(x)dx$.
Since $Q_D$ coincides with some element of $H^{\frac 32}(\partial\Omega)$ on $\Gamma_D$, by standard results on elliptic boundary value problems, there exists an harmonic extension $\tilde{Q}(x)\in H^2(\O;\S_0)$ such that $\tilde{Q}|_{\Gamma_D}=Q_D$ and $\p_n\tilde{Q}|_{\Gamma_N}=0$.

With this notation in place, we seek approximations of the solutions of the
system~\eqref{strongformeqn} of the form
\begin{equation}\label{series}
  u^{(n)}(x,t)=\sum^n_{i=1} d_i(t)v_i(x)\,,
\quad Q^{(n)}(x,t)=\tilde{Q}(x)+\sum^n_{i=1} h_i(t)e_i(x)
\end{equation}
such that %for $k,\,\ell=1,2,\cdots, n$
$(u^{(n)}, Q^{(n)})$ satisfies the generalized Navier-Stokes equations on $V_n$, i.e.,
\begin{align}\label{aproximatea}
    &( \dt{u^{(n)}},v_k)_\O+( (u^{(n)}\cdot \nabla) u^{(n)}, v_k)_\O
+( \nu (Q^{(n)}) \Lsym{u^{(n)}} , \Lsym{v_k})_{\O} \\\nonumber
    &+((\pi_n H(Q^{(n)}) ):\nabla Q^{(n)}, v_k)_\O+\Barg{  (\sigma+\xi\tau_2)(Q^{(n)}, \pi_n H(Q^{(n)}))
-\frac{2\xi}{d}\pi_n H(Q^{(n)}), \nabla v_k }_\O =0
\end{align}
in $(0,T)$
for all $k\in \{1, \ldots, n\}$, the evolution equation for the director field
on $E_n$, i.e.,
\begin{align}\label{aproximateb}
    &\barg{   \dt{Q^{(n)}},e_\ell }_\O +\barg{  (u^{(n)}\cdot\nabla) Q^{(n)}, e_\ell }_\O
-\barg{ S(\nabla u^{(n)},Q^{(n)}),e_\ell }_\O= \barg{ H(Q^{(n)}),e_\ell }_\O
\end{align}
for all $\ell \in \{1, \ldots, n\}$,
and the initial conditions
\begin{align}
    u^{(n)}|_{t=0}&=\mathcal{P}_n u_0\label{appu}\,,\\
    Q^{(n)}|_{t=0}&= \tilde{Q}+\pi_n(Q_0-\tilde{Q})\label{appq}
\end{align}
in $\Omega$. Note that we have replaced the term $-\diverg \tau_1(Q,H(Q))$ in the approximate system by $H(Q):\nabla Q$ because of \eqref{eq:tau1}. In the following we will use that $\partial_t Q^{(n)},\Delta Q^{(n)}\in E_n$ since $\tilde{Q}$ is independent of $t$ and harmonic. Hence $\pi_n\partial_tQ^{(n)}=\partial_tQ^{(n)}$ and $\pi_n\Delta Q^{(n)}=\Delta Q^{(n)}$.

By Lemma \ref{eigenlaplace} and \ref{eigenstokes}, the
above system is well defined %\footnote{for almost every $t$, all the
%terms are integrable in $L^1(\O)$ because $v_k\in W^{1,\infty}(\O)$}
and
can be regarded as a finite-dimensional system of ordinary differential equations which has a solution on a maximal time interval
$[0, T_n)$ with $T_n>0$ for each $n\geq 1$. The following proposition implies that $T_n=\infty$. Moreover, the following a~priori bounds will be essential to pass to the limit $n\to\infty$ in the proof of
Theorem~\ref{weaksolution}.% we state a few results concerning regularity
%and a~priori bounds for the approximation of a solution to~\bref{strongformeqn}.
 \begin{prop}[Lyapunov functional]
\label{aprioriestimate1}~\\
 Let $n\geq 1$. Then  the system \eqref{aproximatea}--\eqref{appq}
has the Lyapunov functional
\begin{equation*}
  E\barg{ u^{(n)}(t,\cdot),Q^{(n)}(t,\cdot) }
=\frac 12\int_{\O}\bbar{u^{(n)}(t,x)}^2\dv{x}
+\mathcal{F}\barg{ Q^{(n)}(t,\cdot) }
\end{equation*}
which satisfies
\begin{equation}\label{energyfunctional1}
  \frac{d}{dt}E\barg{ u^{(n)}(t,\cdot),Q^{(n)}(t,\cdot) }
+ \int_{\O}\nu (Q^{(n)})\bbar{  \Lsym{ u^{(n)} }}^2\dv{x}
+ \int_{\O}\bbar{ \pi_n H(Q^{(n)}) }^2\dv{x}=0\,
\end{equation}
for all $t\in[0,T_n)$. Consequently $T_n=\infty$ for any $n\in\N$.
\end{prop}
\begin{proof}
We multiply \eqref{aproximatea} by $d_k(t)$, integrate in space, and sum over $k=1,\ldots, n$. This is equivalent to
replacing $v_k$  by $u^{(n)}(t)$ and an integration by parts leads to
\begin{equation}\label{enegu}
 \begin{split}
    &\frac 12\frac{d}{dt}\int_{\O}|u^{(n)}|^2\,{\rm d}x +\int_{\O}\nu (Q^{(n)})|\Lsym{u^{(n)}}|^2\,{\rm d}x +
\barg{ \pi_n H(Q^{(n)}): \nabla Q^{(n)}, u^{(n)} }_\O
\\&+
\Barg{  (\sigma+\xi\tau_2)\barg{ Q^{(n)},  \pi_n H(Q^{(n)}) }-
\tfrac{2\xi}{d} \, \pi_nH(Q^{(n)}), \nabla u^{(n)} }_\O =0\,
 \end{split}
\end{equation}
in $[0,T_n)$.
Note that by~\eqref{aproximateb} the boundary conditions for $Q^{(n)}$ are time-independent and
therefore
\begin{align*}
\p_tQ^{(n)}:\p_nQ^{(n)}=0\quad\text{ on } \p\O\,.
\end{align*}
This fact shows in combination with the chain rule and an integration by parts that
\begin{equation*} %\label{energyforqnew}
  \barg{ \dt{Q^{(n)}}, \pi_n H(Q^{(n)}) }_{\O} = \barg{  \dt{Q^{(n)}},  H(Q^{(n)}) }_{\O}
  =-\frac{d}{dt}\mathcal{F} (Q^{(n)} )\,.
\end{equation*}
Consequently we may replace $e_\ell$
in \eqref{aproximateb} by $\pi_nH (Q^{(n)})$ and an integration by parts leads to
\begin{equation}\label{enegq}
   \begin{split}
     \frac{d}{dt}\mathcal{F}(Q^{(n)})&-
\barg{ (u^{(n)}\cdot\nabla) Q^{(n)},\pi_nH (Q^{(n)}) }_{\O} \\
&+ \barg{ S(\nabla u^{(n)},Q^{(n)}),  \pi_n H(Q^{(n)}) }_{\O}
+ \int_{\O}\bbar{  \pi_nH(Q^{(n)}) }^2\,{\rm d}x =0
   \end{split}
 \end{equation}
in $[0,T_n)$.
Moreover, recall that by the algebraic identity \eqref{tensor4}
\begin{multline*} %\label{cancellation}
   \barg{ S(\nabla u^{(n)},Q^{(n)}),  \pi_n H(Q^{(n)}) }_{\O} +
 \barg{ (\sigma+\xi\tau_2)\big(Q^{(n)},\pi_n H(Q^{(n)})\big)-
\frac{2\xi}{d}\pi_n H(Q^{(n)}), \nabla u^{(n)} }_{\O}=0
 \end{multline*}
and this identity implies the assertion of the proposition together with \eqref{enegu}
and \eqref{enegq}.

Finally, \eqref{energyfunctional1} implies that the norm of the solution $(u^{(n)},Q^{(n)})$ cannot blow up in finite time. Hence the characterization of the maximal existence time for solutions of ordinary differential equations yields $T_n=\infty$.
%\textbf{add the passage to $T=\infty$. Yuning:We do not discuss the case $T=\infty$}
\end{proof}

To construct the solution of the system~\bref{strongformeqn} as a weak limit of
approximations we need stronger a~priori estimates concerning regularity in space and time.
The following results hold for all $T>0$.
% \textbf{comment on dependence of constants on $T$.}
Note that, unless otherwise indicated, all constants are generic constant
which may depend on $\Omega$, $T$, $\xi$, $\nu$ and its derivatives, and other parameters of the
system \eqref{strongformeqn} but are independent of $t\in [0,T]$ and
the index $n$ in the approximating system \eqref{aproximatea}--\eqref{appq}.

\begin{prop}[Regularity in space]\label{regspace}
Let $n\in \N$ and let $(u^{(n)},Q^{(n)})$ be the solution of
the system \eqref{aproximatea}--\eqref{appq}. Then
\begin{align*}
 u^{(n)} \in L^2(0,T;\VNSO)\cap L^{\infty}(0,T;L^2_{\sigma}(\O))\,,\quad
Q^{(n)} \in L^2(0,T;H^2(\O))
\end{align*}
and we have the a~priori estimates
\begin{align*}
\|u^{(n)}\|_{L^2(0,T;\VNSO)\cap L^{\infty}(0,T;L^2_{\sigma}(\O))} +
%  &\, \leq C_1\,, \\
%\label{uniformbound3}\\
   \| Q^{(n)} \|_{L^2(0,T;H^2(\O))} &\,\leq C_1(E_0) %(\Omegaprime)
%\label{local}
\end{align*}
where the constant $C_1(E_0)$ is independent of $n$ but depends on $E_0=E(u_0,Q_0)$.
\end{prop}

\begin{proof}
Proposition \ref{aprioriestimate1} implies that the ODE system \eqref{aproximatea}--\eqref{appq}
has a solution
for all times $T>0$ and that the solution satisfies the a~priori estimates
 \begin{subequations}\label{uniformbound}
    \begin{align}
 \sup_{t\in [0,T]}\mathcal{F}(Q^{(n)}(t,\cdot))+\|\pi_n H(Q^{(n)})\|_{L^2(\Omega_T)}&\,\leq  C(E_0)\,,\\
  \|D (u^{(n)}) \|_{L^2(\Omega_T)}+\|u^{(n)}\|_{  L^{\infty}(0,T;L^2_{\sigma}(\O))}&\,\leq C(E_0)
  \end{align}
\end{subequations}
with a constant $C$ independent of $n$.
Korn's inequality implies the bound on $u^{(n)}$ of the proposition.

To prove the bound on $Q^{(n)}$, we need to improve the bound for $\pi_n H(Q^{(n)})$
to a uniform bound for $ H(Q^{(n)})$. Since
\begin{align*}
 \| \pi_n\|_{\mathcal{L}(L^2)} \leq 1 \quad\text{ and }\quad
\pi_n H(Q^{(n)})= \Delta Q^{(n)}+\pi_n L(Q^{(n)})
\end{align*}
these bounds can be obtained from
$\Delta Q^{(n)}$ and $ L(Q^{(n)})$, respectively.
We obtain from \eqref{freeenergy},
 and Young's inequality for almost all $t\in [0,T]$
that
\begin{align*}
 \int_{\O}\left( |\nabla Q^{(n)}(x,t)|^2
+  |Q^{(n)}(x,t)|^2\right)\dv{x}\leq C\(\mathcal{F}(Q^{(n)}(t,\cdot))+1\).
  \end{align*}
As a result
 \begin{equation}\label{uniformbound1}
 \| Q^{(n)}\|_{L^{\infty}(0,T;H^1(\O))}\leq C\,.
 \end{equation}
The definition of $H$ in \eqref{strongforcing}
and (\ref{uniformbound}a) imply
 \begin{equation}\label{difference}
   \norm{ \Delta Q^{(n)} }_{L^2(\Omega_T)}\leq C_1
+\norm{ \pi_n L(Q^{(n)}) }_{L^2(\Omega_T)}\,.
 \end{equation}
Since $L(Q^{(n)})$ contains at most cubic terms in $Q$, we infer from Sobolev's inequality and \eqref{uniformbound1}
that $\|L(Q^{(n)})\|_{L^2(\Omega_T)}\leq C$
and the combination of this bound with \eqref{difference}
leads to
  \begin{equation}\label{uniformbound2}
   \norm{ \Delta Q^{(n)} }_{L^2(\Omega_T)}\leq C
 \end{equation}
and thus
 \begin{equation}\label{a2}
   \|H(Q^{(n)})\|_{L^2(\Omega_T)}\leq C\,.
 \end{equation}
Note that the leading part of $H(Q^{(n)})$ is $\Delta Q^{(n)}$ and therefore we obtain an $H^2$-estimate for  $Q^{(n)}$. %However, due to the boundary
% condition \eqref{mixedboundaryconditions}, the solution $Q^{(n)}$ may lose global
%$H^2$-regularity close to  $\Gamma_D\cap \Gamma_N$
%and in the case that $\Gamma_D\cap \Gamma_N\neq\emptyset$,
%one obtains only interior $H^2$-regularity.
In fact, for all $t\in [0,T]$,
   \begin{equation}
   \norm{ Q^{(n)}(\cdot,t) }_{H^2(\O)}
\leq C \left(\norm{ \Delta Q^{(n)}(\cdot,t) }_{L^2(\O)}+\|Q^{(n)}(\cdot,t)\|_{L^2(\O)}\right).
 \end{equation}
 This estimate combined with \eqref{uniformbound2} gives the second assertion in the proposition.
%\textbf{add details of the $H^2$-estimate since there are no boundary conditions. (Yuning: I changed the def of weak solution, so we do not need boundary $H^2$ regularity of Q to define the Neumann part. The interior $H^2$ is clear, I believe)}
The proof is now complete.
\end{proof}

\begin{prop}[Regularity in time]\label{regtime}
Let $(u^{(n)},Q^{(n)})$ be the solution of
the system \eqref{aproximatea}--\eqref{appq} for some $n\in \N$. Then
%\begin{align*}
% \dt{u^{(n)}} \in L^2(0,T;\mathcal{D}(A^{3/2})')\,,\quad
%\dt{Q^{(n)}} \in L^2(0,T;H^2(\O))
%\end{align*}
%and
we have the a~priori estimate 
\begin{align*}
%\label{ut}
  \norm{ \dt{u^{(n)}} }_{L^2(0,T;\mathcal{D}(A^{3/2})')}+
%\label{qt}
 \norm{  \dt{Q^{(n)}} }_{L^2(0,T;H^{-1}(\O))}\leq C_{2}(E_0)\,,
\end{align*}
where $C_{2}(E_0)$ is independent of $n$.% but may depend on
%$\dist(\Omegaprime,\partial\Omega)$.
\end{prop}
%\textbf{decide on $t$ notation (Yuning: $u_{,t}$ looks strange)}
\begin{proof}
We begin with the estimate for $\dt{u^{(n)}}$.
By Lemma \ref{fractional}, we have the embedding
$\mathcal{D}(A)\hookrightarrow L^{\infty}(\O)$. Thus
\begin{equation}\label{long1}
  \begin{split}
    \bbar{ \barg{ (u^{(n)}\cdot \nabla) u^{(n)}(t), v_k }_\O }
&\leq \|u^{(n)}(t)\|_{L^2(\O)}\|\nabla u^{(n)}(t)\|_{L^2(\O)}\|v_k\|_{L^{\infty}(\O)}
\\&
\leq C\|u^{(n)}(t)\|_{L^2(\O)}\|\nabla u^{(n)}(t)\|_{L^2(\O)}\|v_k\|_{\mathcal{D}(A)}\,,
\end{split}
\end{equation}
for all $t\in [0,T]$ as well as
\begin{equation}
  \bbar{ \barg{ (\pi_n H(Q^{(n)})):\nabla Q^{(n)}(t), v_k }_\O }
\leq C \|H(Q^{(n)}(t))\|_{L^2(\O)}\|\nabla  Q^{(n)}(t)\|_{L^2(\O)}\|v_k\|_{\mathcal{D}(A)} \,,
\end{equation}
and by~\bref{viscosity} %\textbf{state embedding in preliminaries}
\begin{equation}
  \bbar{\barg{ \nu(Q^{(n)})\Lsym{u^{(n)}}(t),\Lsym{v_k }}_\O }
\leq C\|\nabla u^{(n)}(t)\|_{L^2(\Omega)}\|v_k(t)\|_{\mathcal{D}(A^{1/2})}\,.
\end{equation}
Moreover, Lemma \ref{fractional} implies the embedding
$\mathcal{D}(A^{\frac 32})\hookrightarrow W^{1,\infty}(\O)$ and hence
\begin{align*}
\begin{aligned}
&  \bbar{ \barg{  \sigma\big(Q^{(n)}, \pi_n H(Q^{(n)})\big)(t)-
\frac{2\xi}{d} \big(\pi_nH(Q^{(n)})\big)(t), \nabla v_k }_\O }\\
&\qquad   \leq C \barg{1+\|Q^{(n)}(t)\|_{L^2(\O)} } \|H(Q^{(n)}(t))\|_{L^2(\O)}
  \|v_k\|_{\mathcal{D}(A^{3/2})}
\end{aligned}
\end{align*}
and
\begin{align}\label{long2}
\begin{aligned}
& \bbar{ \barg{ \xi\tau_2\big(Q^{(n)},  \pi_n H(Q^{(n)})\big)(t), \nabla v_k } }\\
&\qquad  \leq C \|Q^{(n)}(t)\|_{L^4(\O)} \barg{ \|Q^{(n)}(t)\|_{L^4(\O)}+1 }
\|H(Q^{(n)}(t))\|_{L^2(\O)}\|v_k\|_{\mathcal{D}(A^{3/2})}\,.
\end{aligned}
\end{align}
The combination of \eqref{long1}-\eqref{long2} together with \eqref{aproximatea} yields
for all $k\in\{1,\ldots,n\}$ and almost all $t\in [0,T]$ that
\begin{equation}\label{timeder1}
 \bbar{ \barg{ \partial_t u^{(n)}(t),v_k } }
\leq C b_n(t)\|v_k(t)\|_{\mathcal{D}(A^{3/2})}\,,
\end{equation}
where $b_n(t)$ is defined by
\begin{align*}
 b_n(t)&=\barg{  1+\|u^{(n)}(t)\|_{L^2(\O)} }\|\nabla u^{(n)}(t)\|_{L^2(\O)}  \\
&\qquad
+\|H(Q^{(n)})(t)\|_{L^2(\O)}\|\nabla  Q^{(n)}(t)\|_{L^2(\O)}
+\barg{ 1+\|Q^{(n)}(t)\|^2_{L^4(\O)} } \|H(Q^{(n)})(t)\|_{L^2(\O)}\,.
\end{align*}
In view of  $\left(\partial_t u^{(n)}(t),v\right)_\Omega=0$ for $v\perp V_n$
one obtains
\begin{equation}
  \bbar{ \barg{ \partial_t u^{(n)}(t),v }_{\O} }
\leq C b_n(t)\|v\|_{\mathcal{D}(A^{3/2})}\quad \text{for all } v\in \mathcal{D}(A^{3/2})
\text{ and for almost all }t\in [0,T].
\end{equation}
Thus
\begin{equation}
  \bnorm{ \partial_t u^{(n)}(t) }_{\mathcal{D}(A^{3/2})'}\leq C b_n(t)\quad
\text{ for almost all }\in [0,T]\,.
\end{equation}
The above estimate implies the a~priori bound for $\dt{u}$
% \begin{equation}\label{ut}
%   \bnorm{ \partial_t u^{(n)} }_{L^2(0,T;\mathcal{D}(A^{3/2})')}\leq C_{12}
% \end{equation}
since $b_n(t)$ is in $L^2(0,T)$ due to \eqref{uniformbound} and \eqref{uniformbound1}.

Now we turn to the estimate for $\dt{Q^{(n)}}$.
By H\"{o}lder's inequality, Sobolev's embedding $H^1(\O)\hookrightarrow L^6(\O)$
and \eqref{uniformbound1} we find for almost all $t\in [0,T]$ that
\begin{align*} %\label{a4}
\begin{aligned}
  \bbar{ \barg{  (u^{(n)}\cdot\nabla) Q^{(n)}(t), e_\ell }_{\O} }
&\, \leq \|u^{(n)}(t)\|_{L^3(\O)}\|\nabla Q^{(n)}(t)\|_{L^2(\O)}\|e_\ell\|_{L^6(\O)}\\
&\,  \leq C\|u^{(n)}(t)\|_{H^1(\O)}\|\nabla Q^{(n)}(t)\|_{L^2(\O)}\|e_\ell\|_{H^1(\O)}\, ,
\end{aligned}
\end{align*}
and
\begin{align*}
\begin{aligned} %\label{a5}
  \bbar{ \barg{  S(u^{(n)},Q^{(n)})(t),e_\ell }_{\O} }
& \leq C\|\nabla u^{(n)}(t) \|_{L^2(\O)}
\barg{ \|Q^{(n)}(t) \|^2_{L^6(\O)} +\|Q^{(n)}(t) \|_{L^3(\O)} }
\|e _\ell\|_{L^6(\O)}\\
&  \leq \tilde{C}\|\nabla u^{(n)}(t) \|_{L^2(\O)}
\barg{ \|\nabla Q^{(n)}(t) \|^2_{L^2(\O)}+1 }\|e_\ell\|_{H^1(\O)}\, ,
 \end{aligned}
\end{align*}
as well as
\begin{equation*} %\label{a3}
  \bbar{ \barg{  H(Q^{(n)}(t)),e_\ell }_{\O} }
\leq \|H(Q^{(n)}(t))\|_{L^2(\O)}\|e_\ell\|_{L^2(\O)}\,.
\end{equation*}
These estimates imply together with \eqref{aproximateb}
that for almost every $t\in [0,T]$,
\begin{equation}\label{orth}
 \bbar{\barg{    \dt{Q^{(n)}}(t),e }_{\O} }
\leq C y_n(t)\|e\|_{H^1(\O)},~\forall e\in E_n\,,
\end{equation}
where $y_n(t)$ is defined by
\begin{equation}\label{cn}
  y_n(t)=\|u^{(n)}\|_{H^1(\O)}
\barg{ \|\nabla Q^{(n)}(t)\|^2_{L^2(\O)}+1 }+\|H(Q^{(n)}(t))\|_{L^2(\O)} \,.
\end{equation}
By the orthogonality of the eigenvectors, $(\dt{ Q^{(n)} },e)=0$ for all $e\perp E_n$ and
consequently
\begin{equation}
 \bbar{ \barg{  \p_t Q^{(n)}(t),e }_{\O} }
\leq C y_n(t)\|e\|_{H^1(\O)}\quad \text{ for all }
 e\in H^1(\O;\S_0)\text{ and for almost all } t\in [0,T]\,,
\end{equation}
which leads to
\begin{equation}
 \|  \p_t Q^{(n)} (t) \|_{H^{-1}(\O)}\leq Cy_n(t)\,.
\end{equation}
The assertion of the proposition follows now since
% \begin{equation}\label{qt}
%  \left\|  \p_t Q^{(n)} \right\|_{L^2(0,T;H^{-1}(\O))}\leq C\,.
% \end{equation}
$y_n(t)$ is integrable in $L^2(0,T)$  in view of the estimates in Proposition~\ref{regspace}
and \eqref{cn}, \eqref{uniformbound1} and \eqref{a2}.
\end{proof}

After these preparations we are in a position to give the proof of Theorem~\ref{weaksolution}.

\medskip

\noindent
\textit{Proof of Theorem~\ref{weaksolution}.}
We divide the proof into several steps.

\medskip

\noindent
\textit{Step~1: Compactness and construction of weak limits.}
We conclude from Proposition~\ref{regspace} and~\ref{regtime} the following bounds
on the solutions $(u^{(n)},Q^{(n)})$ of the Galerkin approximation %for all
%open sets $\Omegaprime\Subset\O$,
\begin{align}\label{bounds}
\begin{aligned}
  &\, \norm{ u^{(n)} }_{L^2(0,T;\VNS)\cap L^{\infty}(0,T;L^2_{\sigma})} +
\norm{ \partial_t u^{(n)} }_{L^2(0,T;\mathcal{D}(A^{3/2})')}  \\
&\, \qquad \quad
 +\norm{  \p_t Q^{(n)} }_{L^2(0,T;H^{-1})} +
\|\nabla Q^{(n)}\|_{L^{\infty}(0,T;L^2)}
\\ &\, \qquad \quad
 + \norm{ \Delta Q^{(n)} }_{L^2(\Omega_T)} +
\|H(Q^{(n)})\|_{L^2(\Omega_T)} \leq C\,,
\end{aligned}
\end{align}
where the constant $C$ is independent of $n$.
Moreover,
\begin{align*}
 \norm{  Q^{(n)} }_{L^2(0,T;H^2)} \leq C(\O)\,.
\end{align*}
%\textbf{check why we need this bound (Yuning: we need it to obtain the compactness in the interior)}
By the weak compactness of reflexive Banach spaces
and the weak compactness of the dual spaces of separable spaces we can extract a
%\textbf{was there some weak vs strong measurability issu?}
subsequence of $(u^{(n)},Q^{(n)})$, which we denote again by $(u^{(n)},Q^{(n)})$, such that
the weak convergences
\begin{align}\label{weak}
\begin{aligned}
u^{(n)}& \weakly_{n\to\infty} u &&\text{ in } L^2(0,T;\VNSO)\,,\\
%\textbf{Yuning: not used}~\dt{ u^{(n)} } &\weakly \dt{u} && \text{ in } L^2(0,T;\mathcal{D}(A^{3/2})')\,, \\
Q^{(n)} &\weakly_{n\to\infty} Q && \text{ in } L^2(0,T;H^2(\O;\S_0))\,,\\
%\textbf{Yuning: not used}~\dt{Q^{(n)}} &\weakly \dt{ Q } && \text{ in } L^2(0,T;H^{-1}(\O;\S_0))\,,\\
\Delta Q^{(n)} &\weakly_{n\to\infty} \Delta Q && \text{ in } L^2(\Omega_T;\S_0)\,,\\
\end{aligned}
\end{align}
and the weak-*-convergences
\begin{align}\label{weakstar}
\begin{aligned}
Q^{(n)} & \weaklystar_{n\to\infty} Q &&\text{ in } L^{\infty}(0,T;H^1(\O;\S_0))\,,\\
u^{(n)}& \weaklystar_{n\to\infty} u && \text{ in } L^{\infty}(0,T;L^2_{\sigma}(\O))
\end{aligned}
\end{align}
hold.
Moreover, for fixed $\epsilon>0$ we may choose the subsequence
in view of Lemma $\ref{compact}$ and \eqref{bounds} in such a way that additionally
the strong convergences
\begin{equation}\label{strong}
\begin{aligned}
  Q^{(n)}& \to_{n\to\infty} Q && \text{ in }
L^2(0,T;H^{2-\epsilon}(\O))\cap L^p(\Omega_T),~\forall p\in (1,6)\,,\\
  u^{(n)}& \to_{n\to \infty} u && \text{ in } L^2(0,T;L^2_{\sigma}(\O))\,,
\end{aligned}
\end{equation}
and
\begin{equation}\label{strongc}
\begin{aligned}
  Q^{(n)}& \to_{n\to\infty} Q && \text{ in } C([0,T];L^2(\O))\,,\\
  u^{(n)}&\to_{n\to\infty} u && \text{ in } C([0,T];\VNSprimeO)
\end{aligned}
\end{equation}
hold. The estimates \eqref{strongc}, \eqref{weakstar} and Lemma \ref{temam} imply that
\begin{equation}
 u\in BC_w([0,T];L^2_{\sigma}(\O))\,,\quad
Q\in BC_w([0,T];H^1(\O))\,.
\end{equation}
In order to pass to the limit we assert that the subsequence satisfies additionally
%\textbf{check $d$ vs $3$ issue}
\begin{equation}\label{coincide}
\begin{aligned}
\nu (Q^{(n)}) \Lsym{ u^{(n)} }& \weakly_{n\to\infty} \nu(Q) \Lsym{u} && \text{ in } L^2(\Omega_T;\R^d)\,,\\
H(Q^{(n)})&\weakly_{n\to\infty} H(Q) && \text{ in } L^2(\Omega_T;\S_0)\,.
\end{aligned}
\end{equation}
In fact, for all $\varphi\in L^2(\Omega_T)$ we infer from Lebesgue's
dominated convergence theorem, \eqref{strong} and the properties of the
viscosity coefficient that $\varphi \nu(Q^{(n)})$ converges strongly to
$\varphi{\nu(Q)}$ in $L^2(\Omega_T)$ %\textbf{check}
and the conclusion follows
from the weak convergence of $\Lsym{ u^{(n)} }$ to $\Lsym{u}$ in $L^2(\Omega_T)$.

For the second assertion, note that  by the strong convergence
of $\pi_n$ to the identity map on $L^2(\O;\S_0)$, %\textbf{check!}
we deduce from the third convergence in \eqref{weak} that
$\pi_n \Delta Q^{(n)}\weakly \Delta Q$.
Since $H(Q)=\Delta Q+L(Q)$, where $L(Q)$ is a polynomial
of degree less than or equal to three in $Q$, cf. \eqref{strongforcing},
we only need to show that $L(Q^{(n)})\weakly_{n\to\infty} L(Q)$ weakly in $L^2(\ts{T})$.
However, since $(L(Q^{(n)}))_{n\in\N}$ is bounded in $L^\infty(0,T;L^2(\O))$ due to $H^1(\O)\hookrightarrow L^6(\O)$ and the $L^\infty(0,T;H^1(\O))$-bound for $Q^{(n)}$, $(L(Q^{(n)}))_{n\in\N}$ possesses a weak limit in $L^2(\ts{T})$ (for a suitable subsequence). This weak limit coincides with $L(Q)$ since $L(Q^{(n)})\to_{n\to\infty} L(Q)$ in $L^1(\ts{T})$ because of
\eqref{strong} with $p=3$.% and Sobolev's embedding $H^1(\O)\hookrightarrow L^6(\O)$.

\medskip

\noindent
\textit{Step~2: Derivation of the equation for $u$.}
We replace $v_k$ in \eqref{aproximatea} by $v\in C^1([0,T];W^{1,\infty}(\O))$ of the form
  \begin{equation}\label{combi1}
    v(t)=\sum_{k=1}^Nd^k(t)v_k
  \end{equation}
and obtain the following equation which holds pointwise for $t\in [0,T]$:
\begin{align*}
&   \barg{ \dt{u}^{(n)},v }_\O + \barg{ (u^{(n)}\cdot \nabla) u^{(n)}, v }_\O
+ \barg{  \nu (Q^{(n)}) \Lsym{u^{(n)}} , \Lsym{v} }_\O
+\barg{ (\pi_n H(Q^{(n)}) ):\nabla Q^{(n)}, v }_\O\\
&\qquad +\barg{  (\sigma+\xi\tau_2)(Q^{(n)}, \pi_n H(Q^{(n)}))
-\frac{2\xi}{d}\pi_n H(Q^{(n)}), \nabla v }_\O =0\,.
\end{align*}
If we choose $d^k(t)$ such that $v|_{t=T}=0$ and integrate this equation in time,
then  an integration by parts for the first term yields
\begin{align} \label{weakforu}
\begin{aligned}
\int^T_0\barg{ - \barg{ u^{(n)},\dt{v} }_\O + \barg{ (u^{(n)}\cdot \nabla) u^{(n)}, v }_\O
+\barg{ \nu (Q^{(n)}) \Lsym{u^{(n)}} , \Lsym{v} }_\O }\dv{t}\\
    +\int_0^T \bsqb{ \barg{ (\sigma+\xi\tau_2)(Q^{(n)}, \pi_n H(Q^{(n)}))
-\frac{2\xi}{d}\pi_n H(Q^{(n)}), \nabla v }_\O }\dv{t}\\
    +\int_0^T\barg{ (\pi_n H(Q^{(n)})):\nabla Q^{(n)}, v }_\O \dv{t}
=\left.\barg{ u^{(n)},v }_\O\right|_{t=0}.
 \end{aligned}
\end{align}
By the convergences \eqref{weak}, \eqref{coincide} and \eqref{strong},
one can pass to the limit $n\to \infty$ in the first integral in \eqref{weakforu}.
It remains to show
\begin{align*}
    \int_{\Omega_T} (\sigma+\xi\tau_2)(Q^{(n)}, \pi_n H(Q^{(n)})) :\nabla v\, \dv(x,t)
&\to\int_{\Omega_T} (\sigma+\xi\tau_2)
\left(Q,  H(Q)\right): \nabla v\,\dv(x,t) \,, \\
    \int_{\Omega_T}(\pi_n H(Q^{(n)}) ):\nabla Q^{(n)}\cdot v\,\dv(x,t)
&\to\int_{\Omega_T}  H(Q): \nabla Q\cdot v\,\dv(x,t)
\end{align*}
as $n\to\infty$.
To prove the second assertion, we use \eqref{weak}, \eqref{strong} and \eqref{coincide} to obtain
\begin{alignat*}{2}
\nabla Q^{(n)}\cdot v & \to_{n\to\infty} \nabla Q\cdot v \quad && \text{ in } L^2(\ts{T})\,,\\
\pi_n H(Q^{(n)}) &\weakly_{n\to\infty} H(Q) && \text{ in } L^2(\ts{T}) \,.
\end{alignat*}
Using the strong convergence of $(Q^{(n)})_{n\in\N}$ in $L^4(\ts{T})$ and the weak convergence of $H(Q^{(n)})$ in $L^2(\ts{T})$ one can easily prove the first assertion since all terms in $\tau_2(Q,H)$ and $\sigma(Q,H)$ are linear with respect to  $H$ and at most quadratic with respect to $Q$.
Hence we conclude
\begin{equation*}
\begin{split}
  &\int_{\Omega_T}\left(-u\cdot \dt{v}+(u\cdot\nabla) u\cdot v
+\nu(Q)  \Lsym{u}:\Lsym{v}+  H(Q):\nabla Q\cdot v\right)\,\dv(x,t)\\
&+\int_{\Omega_T}\left(\barg{ \sigma+\xi\tau_2 }(Q,H(Q))
-\frac{2\xi}{d}H(Q)\right):\nabla v \dv{(x,t)}
=\int_{\O}u_0(x)\cdot v(0,x)\dv{x}
\end{split}
\end{equation*}
for any $v$ of the form \eqref{combi1} with $v(T,\cdot)=0$.
By a density argument, the above equation also holds for any
$v\in C_0^1([0,T);V(\O))$.
This equation together with \eqref{eq:tau1} implies the weak formulation \eqref{weakforufield}.
% In order to prove that $u$ satisfies  the weak formulation \eqref{weakforufield}, it remains to  show that
% \begin{equation*}
%  -\int_{\Omega_T}\div \tau_1(Q)\cdot v\dv{(x,t)}= \int_{\Omega_T}H(Q):\nabla Q\cdot v\dv{(x,t)}\,.
% \end{equation*}
% However, since $v$ is divergence-free, \eqref{eq:tau1} implies
% \begin{align*}
%  -\int_{\Omega_T}\div \tau_1(Q)\cdot v\dv{(x,t)}
% %&\, =\int_{\Omega_T}\p_j\barg{ \p_iQ_{\a\b}\p_jQ_{\a\b} } v_i\dv{(x,t)}
% %=\int_{\Omega_T} \div(\nabla Q\cdot \nabla Q)\cdot v\dv{(x,t)}
% %\\
% %&\,  =\int_{\Omega_T}\Delta Q:\nabla Q\cdot v\dv{(x,t)}
% %+\underbrace{\int_{\Omega_T}(L(Q):\nabla Q)\cdot v\dv{(x,t)}}_{=0}\\
% &\, =\int_{\Omega_T}(H(Q):\nabla Q)\cdot v\dv{(x,t)}.
% \end{align*}
% \textbf{triple check}

\medskip

\noindent
\textit{Step~3: Derivation of the equation for $Q$. }% and verification of \eqref{weakbc}.}
We replace $e_\ell$ in \eqref{aproximateb} by $\Psi\in C^1([0,T];H^1(\O;\S_0))$
of the form $\Psi(t)=\sum_{\ell=1}^N d^\ell(t)e_\ell$,
integrate in time on $[0,T]$ and integrate by parts in the first term. This yields
\begin{align*}
  -\int_{\Omega_T}  Q^{(n)}:\dt{\Psi} \dv{(x,t)}
+\int_{\Omega_T} (u^{(n)}\cdot\nabla) Q^{(n)}:\Psi\dv{(x,t)}-\int_{\Omega_T} S(\nabla u^{(n)},Q^{(n)}):\Psi\dv{(x,t)}\\
=\int_{\Omega_T}  H(Q^{(n)}):\Psi\dv{(x,t)}+\left.\barg{ Q^{(n)},\Psi }_\Omega\right|_{t=0}.
\end{align*}
Employing \eqref{weak} and \eqref{strong} we conclude
\begin{align*}
  \int_{\Omega_T} S(\nabla u^{(n)},Q^{(n)}):\Psi\dv{(x,t)}\to_{n\to\infty} \int_{\Omega_T}S(\nabla u,Q):\Psi\dv{(x,t)}.
\end{align*}
Hence we can pass to the limit in the equation above.
Through a density argument we obtain the weak formulation \eqref{weakforq}. Finally, the boundary conditions \eqref{mixedboundaryconditions} (for almost every $t$) follow from the fact that $(u^{(n)},Q^{(n)})$ satisfy these boundary conditions and the (weak) continuity of the Dirichlet and Neumann trace operators on $H^1(\Omega)$, $H^2(\Omega)$, respectively.
% Now we verify \eqref{weakbc}.
% By Green's formula and \eqref{series}:
%  \begin{equation*}
%       \begin{split}
%         &\int_{\O}\(\Delta Q^{(n)}:\Psi+\nabla Q^{(n)}:\nabla\Psi\)\\
%         =&\int_{\p\O}\p_nQ^{(n)}:\Psi
%         =\int_{\Gamma_N}\(\p_n\tilde{Q}(x)+\sum^n_{i=1} h_i(t)\p_n e_i(x)\):\Psi=0.
%       \end{split}
%     \end{equation*}
%     In the last step, we used that fact that the normal derivative of $e_i$ and $\tilde{Q}$  vanishes on $\Gamma_N$ and $\Psi|_{\Gamma_D}=0$. Combining the above result with   \eqref{weak} and \eqref{weakstar} yields \eqref{weakbc}.
%%%%%%%%%%%%%%%%%%%%%%%%%%%%%%%%%%%%%%%%%%%%%%%%%%%%%%%%%%%%%%
%%                                                          %%
%% R E G U L A R I T Y  I N  T I M E                        %%
%%                                                          %%
%%%%%%%%%%%%%%%%%%%%%%%%%%%%%%%%%%%%%%%%%%%%%%%%%%%%%%%%%%%%%%

\section{Regularity in time and proof of Theorem~\ref{localstrong}}

The proof of a unique local solution with additional regularity in time is obtained by
Banach's fixed-point theorem. In Section~\ref{fasetup} we define the
function spaces and the operators to which we will apply the
fixed-point theorem, in Section~\ref{linearstuff} we prove that
the linear operator $\mathcal{L}:X_0\to Y_0$ defined in \bref{linop} is bounded,
onto and one-to-one, in
Section~\ref{nonoinearstuff} we verify that the nonlinear operator $\mathcal{N}_0$
in \bref{newnon} is locally Lipschitz continuous with small Lipschitz constant for
$T$ sufficiently small, and in Section~\ref{finallyproof} we give the proof of
Theorem~\ref{localstrong}.
In this section we assume that $(u_0,Q_0)\in \Zspace$. As usual, we
formulate the first equation in \eqref{strongformeqn} weakly by testing with divergence free vector fields. Then we obtain
\begin{align}\label{pressure}
 \dt{u} - P_{\sigma}\div(\nu(Q) \Lsym{u})
&\, =P_{\sigma}\div\big(\tau(Q,H(Q))+\sigma(Q,H(Q))-u\otimes u\big)\,,\\
\dt{Q}-\Delta Q &\, = -(u\cdot \nabla) Q +S(\nabla u,Q)+L(Q)\,,
\end{align}
where $P_\sigma\colon H^{-1}(\Omega;\R^d)\to H^{-1}_\sigma(\Omega)$ and $\div\colon L^2(\Omega;\R^{d\times d})\to H^{-1}(\Omega;\R^d)$ are defined as in Section~\ref{sec:Notation}.
%%Once we have a solution of this system, we obtain a solution of \bref{strongformeqn}.

\subsection{Function spaces and operators}\label{fasetup}
The idea is to rewrite the nonlinear system~\bref{strongformeqn} as an
operator equation between suitable Banach spaces.
% see also \eqref{stationaryequ} and \eqref{trace} below
% for the definition of the operators.
We begin with the definition of the linear and the nonlinear operator
in this fixed-point formulation
and use these definitions together with the regularity in time
asserted in Theorem~\ref{localstrong} as motivation for the definition of the
function spaces for the domain and the range of the operators.
% In order to obtain a fix point formulation, w
We linearize the system about the constant trajectory $Q_0$
of the $Q$-tensor. Then the principal part of the linear system is given by  $\mathcal{S}$ and $\mathcal{L}$, where
 \begin{equation}\label{stationaryequ}
  \mathcal{S}(Q_0)\begin{pmatrix} u\\ Q\end{pmatrix}=
\begin{pmatrix}
  P_{\sigma}\div\bsqb{ \nu(Q_0)\Lsym{u}
+ (\sigma+\xi\tau_2)(Q_0,\Delta Q)-\frac{2\xi}{d}\Delta Q } \\
\Delta Q+S(\nabla u,Q_0)
\end{pmatrix},
\end{equation}
and
\begin{equation}\label{linop}
\mathcal{L} (Q_0)\begin{pmatrix} u\\ Q \end{pmatrix}=
\frac{d}{dt}\begin{pmatrix} u\\ Q \end{pmatrix}
-\mathcal{S}(Q_0)\begin{pmatrix} u\\ Q\end{pmatrix}\,,
\end{equation}
respectively. 

 As a result, 
we can consider all the terms in \eqref{pressure} as a functional over $H_{0,\sigma}^1(\O)$ and once we obtain a solution to \eqref{pressure}, we can disregard the $P_{\sigma}$ in \eqref{pressure} by adding a pressure term $\nabla p$ due to standard results.
The nonlinear operator $\mathcal{N}$ in the reformulation of the system of partial
differential equations as the operator equation
$\mathcal{L}(Q_0)(u,Q) = \mathcal{N}(Q_0)(u,Q)$ is given by
\begin{equation*}
\begin{split}
  &\mathcal{N}(Q_0)\begin{pmatrix}
u\\
Q
\end{pmatrix}=
\begin{pmatrix}
P_{\sigma}\div\left[(\nu(Q)-\nu(Q_0))\Lsym{u}+\tau_1(Q)-u\otimes u\right]\\
-(u\cdot \nabla) Q -L(Q)
\end{pmatrix}\\
&+
\begin{pmatrix}
P_{\sigma}\div\bsqb{ (\sigma+\xi\tau_2)(Q,\Delta Q)-(\sigma+\xi\tau_2)(Q_0,\Delta Q)
+(\sigma+\xi\tau_2)(Q,L(Q))-\frac {2\xi}{d}L(Q) }\\
S_1(\nabla u,Q)-S_1(\nabla u,Q_0)+\xi S_2(\nabla u,Q)-\xi S_2(\nabla u,Q_0)
\end{pmatrix}\,.
\end{split}
\end{equation*}
It is also useful to pass to a formulation with homogeneous initial and boundary
conditions.
% Again we will define the precise regularity later.
% Suppose that $(u_0, Q_0)\in \Zspace$ given by the initial conditions~\bref{initialconditions}
% satisfies the mixed boundary conditions~\bref{mixedboundaryconditions}.
Note that \eqref{strongformeqn} together with the initial and
boundary conditions can be
formulated by the operator equation
\begin{equation*}%\label{oper3}
  \mathcal{L}(Q_0)\begin{pmatrix} \uh+u_0\\ \Qh+Q_0 \end{pmatrix}
=\mathcal{N}(Q_0)\begin{pmatrix} \uh+u_0\\ \Qh+Q_0 \end{pmatrix}\,,
% \quad \begin{pmatrix} \uh \\ \Qh \end{pmatrix}\in X_0.
 \end{equation*}
where $(\uh, \Qh)$ satisfies the corresponding homogeneous initial and
boundary conditions. 	
By the definition of the linear operator $\mathcal{L}$ in \eqref{linop},
the above identity is equivalent to
\begin{equation}\label{trans}
  \mathcal{L}(Q_0)
\begin{pmatrix} \uh\\ \Qh \end{pmatrix}=
\mathcal{N}(Q_0)\begin{pmatrix} \uh+u_0\\ \Qh+Q_0 \end{pmatrix}+
\mathcal{S}(Q_0)\begin{pmatrix} u_0\\ Q_0 \end{pmatrix}
% \in Y_0,\quad
% \begin{pmatrix} \uh\\ \Qh \end{pmatrix}\in X_0
 \end{equation}
and the right-hand side defines a nonlinear operator
\begin{equation} \label{newnon}
  \mathcal{N}_0(Q_0)\begin{pmatrix}
\uh\\
\Qh
\end{pmatrix}=\mathcal{N}(Q_0)\begin{pmatrix}
\uh+u_0\\
\Qh+Q_0
\end{pmatrix}+ \mathcal{S}(Q_0)\begin{pmatrix}
u_0\\
Q_0
\end{pmatrix}.
\end{equation}
We now turn to the definition of functions spaces $X_0$ and $Y_0$ such that
$\mathcal{L},\,\mathcal{N}_0:X_0\to Y_0$ with $\mathcal{L}$ an isomorphism.
Motivated by the idea to construct solutions which are
twice differentiable in time and
the precise assertions in Theorem~\ref{localstrong},
we  define the function space
for the range of the operators by
\begin{align*}
  Y_u =H^1(0,T;\VNSprimeO),\quad  Y_Q=H^1(0,T;L^2(\O;\S_0))\,.
\end{align*}
In particular, we need to prove regularity of solutions of the linear equation
$\mathcal{L}(Q_0)(\uh,\Qh) = (f,g)$ with right-hand side $(f,g)\in Y_0$ subject to
homogeneous initial data. The general linear theory requires a compatibility condition
which is taken care of by the definition
% \textbf{is that true, needs $S_0$, define here $H$, see below????}
of $Y_0$ as
\begin{align} \label{Y0}
  Y_0&=\left\{ (f,g)\in Y_u\times Y_Q:
(  f,  g)|_{t=0}\in \HNSO
\times H^1_{\Gamma_D}(\O) \right\}\,.
\end{align}
These spaces are equipped with the usual norms in product spaces and for spaces
of functions of one variable with values in a Banach space together with the
correct norm of the initial data. 
More precisely, the norm of $Y_0$ is given by
\begin{equation}\label{norm}
  \|(f,g)\|_{Y_0}
=\left(\|(f,g)\|_{Y_u\times Y_Q}^2+\|(f,g)|_{t=0}\|_{\HNSO\times H^1(\O)}^2\right)^{\frac12}\,.
\end{equation}
Note that the second part of the norm is not
controlled by trace theorems applied to $Y_u\times Y_Q$.
The domains of the operators are given by the Banach spaces
\begin{alignat*}{3}
X^1_u&=  H^2(0,T;\VNSprimeO)\,,\quad
 &X^2_u&=H^1(0,T;\VNSO)\,,\quad &X_u&=X^1_u \cap X_u^2\,,\\
X^1_Q&= H^2(0,T;L^2(\O;\S_0))\,, &X^2_Q&= H^1(0,T;H^2(\O;\S_0))\,, &X_Q&=X_Q^1\cap X_Q^2
\end{alignat*}
together with the norms
\begin{align}\nonumber
  \|u\|_{X_u}&=\left(\|u\|_{X_u^1}^2+\|u\|_{X_u^2}^2+\|u|_{t=0}
  \|_{\VNSO}^2+\|u_t|_{t=0}
  \|_{L^2(\O)}^2\right)^{\frac12}\,,
\\ \label{XQ}
  \|Q\|_{X_Q}&=\left(\|Q\|_{X_Q^1}^2+\|Q\|_{X_Q^2}^2+\|Q|_{t=0}\|_{ H^2(\O)}^2+\|\dt{Q}|_{t=0}\|_{ H^1(\O)}^2\right)^{\frac12}\,.
\end{align}
Note that the last two terms in the norms are important to obtain in the subsequent constants that are uniformly bounded as $T\to 0$, cf. e.g. \eqref{interpo2}.
The corresponding subspaces related to the homogeneous initial and boundary conditions in
the formulation of the problem are defined by
\begin{align}\label{X0}
X_0&= \left\{(u,Q)\in X_u\times X_Q:\mathcal{T}(Q)=
\left( 0,0 \right),\left(u,Q \right)|_{t=0}
=\left(0,0 \right)\right\}\,.
\end{align}
Here the trace operator $\mathcal{T}(Q)$ is given by
\begin{equation}\label{trace}
  \mathcal{T}(Q)=
\left( Q|_{(0,T)\times \Gamma_D}, \p_n Q|_{(0,T)\times\Gamma_N}  \right)\,,
\end{equation}
% where we do not indicate the dependence of $\mathcal{T}$ on the time horizont $T$.
and $X_0$ is equipped with the product norm
\begin{equation*}
  \|(u,Q)\|_{X_0}=\|(u,Q)\|_{X_u\times X_Q}\,.
\end{equation*}
Together with these norms the space $X_0$ and $Y_0$ are closed subspaces of the spaces
$X_u\times X_Q$ and $Y_u\times Y_Q$, respectively.

One can check  the compatibility condition in $\Zspace$ that
the right-hand side of \eqref{trans} belongs to $Y_0$ if $(u_h,Q_h)\in X_0$, cf. the proof of Proposition~\ref{nonlinearterms} (i) below.

\subsection{Existence and uniqueness for the linear system}\label{linearstuff}
The key point in the proof of the local existence of solutions with additional
regularity in time is the verification of global solvability of the linear
system and of its regularity properties. This is achieved based on results
on abstract parabolic evolution equations which we recall for the convenience of
the reader.
Suppose that $\VWloka$ and $\HWloka$ are two separable Hilbert spaces such that
the embedding $\VWloka\hookrightarrow \HWloka$
is injective, continuous, and dense.
Fix $T\in (0,\infty)$. Suppose that for all $t\in [0,T]$ a bilinear
form $a(t;\cdot, \cdot): \VWloka\times \VWloka\to \R$ is given which satisfies
for all $\phi$, $\psi\in \VWloka$ the following assumptions:
\begin{itemize}
 \item [(a)] $a(\cdot ;\phi, \psi)$ is measurable on $[0,T]$;

 \item [(b)] there exists a constant $c>0$, independent of $t, \phi$  and $\psi$, with
\begin{align*}
 \bbar{ a(t;\phi, \psi) } \leq c \| \phi\|_\VWloka\| \psi\|_\VWloka
\quad\text{ for all }t\in [0,T];
\end{align*}

\item [(c)] there exist $k_0$, $\alpha\geq 0$ independent of $t$ and $\phi$,
with
\begin{align*}
  a(t;\phi, \phi) + k_0 \| \phi\|_\HWloka^2 \geq \alpha \| \phi\|_\VWloka^2
\quad \text{ for all }t\in [0,T]\,;
\end{align*}

 \item [(d)] $a(\ \cdot\ ;\phi, \psi)$ is differentiable,
$a(\ \cdot \ ;\phi, \psi)$ is continuous in $[0,T]$ and
$\partial_t a(t;\phi, \psi)$
is measurable with $|\partial_t^j a(t;\phi, \psi)|
\leq c\| \phi\|_\VWloka\| \psi\|_\VWloka$
for $j=0$, $1$ with $c$ independent of $t$.
\end{itemize}

\begin{theorem}\label{Wlokaregularity}
Suppose that (a)--(c) hold.
Then there exists a representation operator $L(t): \VWloka \to \VWloka^\prime$
with $a(t;\phi, \psi) = \scp{ L(t)\phi}{\psi }_{\VWloka^\prime,\VWloka}$,
which  is continuous and linear for fixed $t$. Moreover, for all
$f\in L^2((0,T);\VWloka^\prime)$ and $y_0\in \HWloka$, there exists a
unique solution
\begin{align*}
 y\in \bset{ v:[0,T]\to \VWloka \text{ with }v\in L^2(0,T;\VWloka),\,
\dt{v}\in L^2(0,T;\VWloka^\prime) }
\end{align*}
which solves the equation
\begin{align*}
% \frac{ {\rm d}y }{ {\rm d}t }
\dt{y}+  L(t) y = f \quad \text{ in } \VWloka^\prime\ \text{for a.e. }t\in (0,T)\,,
\end{align*}
subject to the initial condition $y(0) = y_0$.
Finally, assume additionally that (d) holds and that $y_0 \in \VWloka$. Then
$L: \SonetwoWloka((0,T);\VWloka) \to  \SonetwoWloka((0,T);\VWloka^\prime)$
is continuous and for all $f\in \SonetwoWloka((0,T);\VWloka^\prime)$
which satisfy the compatibility condition $f(0)-L(0)y_0\in \HWloka$ %% \VWloka$
the solution $y$ satisfies
\begin{align*}
 y\in \SonetwoWloka((0,T);\VWloka)\quad \text{ and }\quad
 \dtt{y}\in L^2\barg{(0,T);\VWloka^\prime}\,.
\end{align*}

\end{theorem}
The proof of this theorem can be found in  \cite[Lemma 26.1 and Theorem 27.2]{WlokaPDE}.

The following result establishes the invertibility of the linear operator equation.
Note that we are seeking a solution of the linear equation in $X_0$, i.e., a
solution with homogeneous initial and boundary conditions.

\begin{prop}[Homogeneous linear system]
\label{linearizedprop}
Let $T\in (0,1]$. Then $\mathcal{L}:X_0\to Y_0$, and
for every $(f,g)\in Y_0$, the operator equation
\begin{align*}
   \mathcal{L}(Q_0) (u,Q) = (f,g)
\end{align*}
has a unique solution $(u,Q)\in X_0$ satisfying
\begin{equation}\label{stronginverse}
\| \mathcal{L}^{-1}(Q_0)(f,g)\|_{X_0}=
\|(u,Q)\|_{X_0}\leq C_{\mathcal{L}}\|(f,g)\|_{Y_0}
\end{equation}
where $C_{\mathcal{L}}$ is independent of $T\in(0,1]$ and $(f,g)\in Y_0$.
In particular $\mathcal{L}(Q_0):X_0\to Y_0$ is invertible and
$\mathcal{L}^{-1}(Q_0)$ is a bounded linear operator with norm independent of $T\in(0,1]$.
\end{prop}

\begin{proof}
The idea is to apply Theorem~\ref{Wlokaregularity} and we carry out this program
in the subsequent steps.

\medskip

\noindent
\textit{Step 1: Function spaces.} Since the regularity is in time, we only
need to incorporate the regularity in space into the spaces $\VWloka$ and $\HWloka$.
We define the Hilbert spaces
\begin{align*}
 \HWloka  = \HWloka_1 \times \HWloka_2
& =\HNSO\times H^1_{\Gamma_D}(\Omega;\S_0) \,,\\
\VWloka   = \VWloka_1 \times \VWloka_2
& = \VNSO \times
\bset{ Q\in H^2(\Omega;\S_0)\cap H^1_{\Gamma_D}(\Omega;\S_0),
\partial_n Q|_{\Gamma_N}=0 }
\end{align*}
and equip them with the inner products
\begin{align*}
  ((u,Q),(v,P))_\HWloka
&\, = (u,v)_{L^2(\Omega)}+ (Q,P)_{H^1(\Omega)}\quad\text{ for all }(u,Q),(v,P)\in \HWloka\,,\\
((u,Q),(v,P))_\VWloka
&\,=(u,v)_{H^1(\Omega)} + (Q,P)_{H^2(\Omega)}\quad\text{ for all }(u,Q),(v,P)\in \VWloka\,.
\end{align*}
Here the inner product in the Sobolev spaces
$H^k$, $k\geq 1$, is the usual inner product in these spaces. The spaces
$\HWloka$ and $\VWloka$ are Hilbert spaces. 

Recall that $\VWloka_2\hookrightarrow \HWloka_2\cong \HWloka_2' \hookrightarrow \VWloka_2'$, where $\HWloka_2$ is identified with $\HWloka_2'$ via the Riesz isomorphism
$P\mapsto (\nabla P,\nabla \cdot)_{L^2(\Omega)}+ (P,\cdot)_{L^2(\Omega)}$. This implies that
for all $P,\Phi\in \VWloka_2$
\begin{align}\label{indentifyprod}
\weight{P,\Phi}_{\VWloka_2',\VWloka_2}= (P,\Phi)_{\HWloka_2}
= \int_{\Omega} \nabla P:\nabla \Phi \dv{x} + \int_{\Omega} P: \Phi \dv{x}\, 
= \int_\Omega P:(I- \Delta) \Phi\dv{x}.
\end{align}

\medskip

\noindent
\textit{Step 2: Operators and bilinear forms.} The most subtle point is the correct
definition of the bilinear form $a$ since the natural bilinear form
associated with $\mathcal{S}$ given by 
\begin{align*}
\bscp{ L(v, P)}{(\varphi,\Phi)}_{\VWloka',\VWloka}
=& \int_\Omega \nu(Q_0)\Lsym{v}:\Lsym{\varphi}\dv{x}
  +\int_\Omega ((\sigma+\xi \tau_2)(Q_0,\Delta P)
-\tfrac{2\xi}{d} \, \Delta P):\nabla{\varphi}\dv{x}\\
&\qquad - \int_\Omega (\Delta P + S(\nabla v,Q_0)):\Phi \dv{x}
\end{align*}
does not lead to a bilinear form which is coercive on $\VWloka$. This
is achieved by taking advantage of the cancellation property in \eqref{tensor4} and
by defining a bilinear form which is independent of time by
\begin{align*}
\bscp{ \tildeL (v, P)}{(\varphi,\Phi)}_{\VWloka',\VWloka}
=& \int_\Omega \nu(Q_0)\Lsym{v}:\Lsym{\varphi}\dv{x}
  +\int_\Omega ((\sigma+\xi \tau_2)(Q_0,\Delta P)-\tfrac{2\xi}d\Delta P)
:\nabla{\varphi}\dv{x}\\
&\qquad - \int_\Omega (\Delta P + S(\nabla v,Q_0)):(I-\Delta)\Phi \dv{x}
\end{align*}
for all $(v,P),(\varphi,\Phi)\in \VWloka$. The additional term in the equation
has to be compensated for on the right-hand side of the linear system and therefore
we associate to $(f,g)\in Y_0$ the element
$(F,\tildeG)\in L^2\barg{ (0,T);\VWloka^\prime }$ by
\begin{align*}
 \scp{(F(t),\tildeG(t))}{(\phi, \Phi)}_{\VWloka',\VWloka}
=\int_\Omega \barg{ f(t)\cdot \phi + g(t):(I-\Delta)\Phi }\dv{x}
\end{align*}
for all $(\phi, \Phi)\in \VWloka$ and almost all $t\in (0,T)$.
% for all $(\phi, \Phi)\in L^2\barg{ (0,T);\VWloka}$.
We now assert that
the solution of the abstract evolution equation
\begin{align}\label{eq:AbstractEq}
 \bscp{(\dt{u},\dt{Q})}{(\phi,\Phi)}_{\VWloka',\VWloka}
+\bscp{ \tildeL (u, Q)}{(\varphi,\Phi)}_{\VWloka',\VWloka}
=  \scp{(F,\tildeG)}{(\phi, \Phi)}_{\VWloka',\VWloka}
\end{align}
for all $(\varphi,\Phi)\in \VWloka$
subject to the initial condition $(u(0), Q(0))=(0,0)\in \HWloka$
is indeed a weak solution of the linear evolution equation.
% Finally, \eqref{eq:AbstractEq} is equivalent to $\mathcal{L}(Q_0)
% \begin{pmatrix}
%   u\\ Q
% \end{pmatrix} =
% \begin{pmatrix}
%   F\\ G
% \end{pmatrix}
% $.
The choice of $(\phi,0) \in \VWloka$ implies the correct equation for $u$.
To identify the equation for $Q$, choose $(0,\Phi) \in \VWloka$ as
test function and obtain  by~\bref{indentifyprod}
\begin{align*}
% \MoveEqLeft
\int_\Omega g(x,t):(I-\Delta) \Phi(x)\dv{x}
 &=
\weight{\dt{Q},\Phi}_{\VWloka_2',\VWloka_2}
- \int_\Omega (\Delta Q+S(\nabla u,Q_0)): (I- \Delta) \Phi \dv{x}
\\ &=  \int_\Omega (\dt{Q}- \Delta Q-S(\nabla u,Q_0)):(I- \Delta) \Phi \dv{x}\,.
\end{align*}
Since $(I-\Delta)\colon \VWloka_2\to L^2(\Omega;\S_0)$ is bijective,
cf. e.g.~\cite[Theorems 4.10 and 4.18]{McLean},
we conclude
\begin{equation*}
  \partial_t Q -\Delta Q -S(\nabla u,Q_0) = g \qquad \text{a.e. in } \Omega\times (0,T).
\end{equation*}

\medskip

\noindent
\textit{Step 3: Existence of time-regular solutions.}
The existence of time-regular solutions follows from Theorem~\ref{Wlokaregularity}
once we have verified the assumptions (a)--(d) on the bilinear form
and the regularity assumptions on the right-hand side.
Since $a$ does not depend on time, (a) and (d) are immediate. By Sobolev's embedding theorem,
$H^2\hookrightarrow C^0$, hence $Q_0\in L^\infty$ and (b) follows
from~\bref{viscosity} and H\"older's inequality.
Moreover, in view of the cancellation property \eqref{tensor4}, Korn's inequality and
Young's inequality,
\begin{align*}
\bscp{ L(v, P)}{(v, P)}_{\VWloka',\VWloka}
   &= \int_\Omega \nu(Q_0)\Lsym{v}:\Lsym{v}\dv{x} + \int_\Omega |\Delta P|^2 \dv{x}-
\int_\Omega (\Delta P + S(\nabla v,Q_0)):P \dv{x}\\
&\geq c_0\|(v,P)\|_{\VWloka}^2 - C\|(v,P)\|_{\HWloka}^2
\end{align*}
for all $(v,P)\in \VWloka$, $t\in [0,T]$, and suitable constants $c_0$, $C>0$.
Therefore (c) is satisfied. Finally we obtain
by the definition of $Y_0$ that $(f,g) \in Y_0$ is equivalent to
$(F,\tildeG)\in \SonetwoWloka(0,T;\VWloka^\prime)$ and $(F(0), G(0)) \in \HWloka'\cong\HWloka$.
Hence there exists a unique solution
$(u,Q)\in \StwotwoWloka(0,T;\VWloka')\cap \SonetwoWloka(0,T;\VWloka)$ of the
abstract evolution equation
and therefore for the linear equation. %\textbf{check $G(0)$ in the right space}

\medskip

\noindent
\textit{Step 4: $\mathcal{L}$ is a bounded isomorphism.}
The only regularity statement which does not follow from the regularity
of the solution in Step~3 is the assertion $Q\in H^2(L^2)$. Note that the
right-hand side in the equation for $Q$ belongs to $H^1(L^2)$. Therefore
$\dt{Q}\in H^1(L^2)$ and $Q\in H^2(L^2)$.

Altogether, we have proven that $\mathcal{L}(Q_0)\colon X_0\to Y_0$
is an isomorphism.
The boundedness of the operator norm of
$\mathcal{L}(Q_0)^{-1}\colon Y_0\to X_0$ uniformly in $0<T\leq 1$ 
can be shown as follows. By a standard energy estimate,
i.e., by taking the duality product of \eqref{eq:AbstractEq} and $(u,Q)^T$
and integration in time, we derive 
\begin{equation*}
  \sup_{0\leq t\leq T} \|(u(t),Q(t))\|_{\HWloka}^2
+ c_0\int_0^T\|(u(t),Q(t))\|_\VWloka^2 \sd t \leq C\|(F,G)\|_{L^2(0,T;\VWloka')}^2
\end{equation*}
with constants $c_0,C$ independent of $T>0$. Moreover, if we
differentiate \eqref{eq:AbstractEq} with respect to $t$ and
take the duality product with $(\partial_t u,\partial_t Q)^T$, then we discover
\begin{eqnarray*}
  \lefteqn{\sup_{0\leq t\leq T} \|(\partial_t u(t),\partial_t Q(t))\|_{\HWloka}^2
 + c_0\int_0^T\|(\partial_t u(t),\partial_t Q(t))\|_\VWloka^2 \sd t}\\
 &\leq& C\barg{ \|(\partial_t F,\partial_t G)\|_{L^2(0,T;\VWloka')}^2
 + \|(\partial_t u(0), \partial_t Q(0))\|_{\HWloka} } . %\\
%&& \left. + \left|\int_0^T \left\langle \left(\frac{d}{dt}A(t)\right)
%        \begin{pmatrix}
%          u(t)\\ Q(t)
%        \end{pmatrix},
%        \begin{pmatrix}
%          \partial_tu(t)\\ \partial_tQ(t)
%        \end{pmatrix}\right\rangle_{\VWloka',\VWloka}\right|\right)
\end{eqnarray*}
By the previous estimate, Young's inequality, and \eqref{eq:AbstractEq} for $t=0$ we conclude
\begin{eqnarray*}
  \lefteqn{\sup_{0\leq t\leq T} \|(\partial_t u(t),\partial_t Q(t))\|_{\HWloka}^2
 + c_0\int_0^T\|(\partial_t u(t),\partial_t Q(t))\|_\VWloka^2 \sd t}\\
 &\leq& C\barg{ \|(F, G)\|_{H^1(0,T;\VWloka')}^2 + \|(F(0), G(0))\|_{\HWloka}^2 }
= C\|(F,G)\|_{Y_0}^2
\end{eqnarray*}
for all $0<T\leq 1$. Finally,
second order time derivatives of \eqref{eq:AbstractEq} imply
the same estimate for $(\partial_t^2 u,\partial_t^2 Q)\in L^2(0,T;\VWloka')$.
The foregoing estimates can be summarize by
\begin{equation*}
  \|(u,Q)\|_{X_0} \leq C\|(F,G)\|_{Y_0}
\end{equation*}
for all $T\in (0, 1]$.
\end{proof}

\subsection{Local Lipschitz continuity of the nonlinear operator}\label{nonoinearstuff}

In this section we analyze the nonlinear terms.
The fundamental properties of the nonlinear operator are summarized
in the following proposition.

\begin{prop}\label{nonlinearterms}
Fix $0<T\leq 1$, $R>0$, $(u_0, Q_0)\in \Zspace$,
let $\mathcal{N}_0(Q_0)$ be the nonlinear operator defined
in \eqref{newnon}, and recall that
$\BXzeroR =\{(v,P)\in X_0,\, \|(v,P)\|_{X_0}\leq R\}$.
Then the following assertions are true
for all $(\uhi,\Qhi)\in \BXzeroR$, $i=1,\,2$:

\begin{itemize}
 \item [(i)] $\mathcal{N}_0(Q_0)$ maps $X_0$ to $Y_0$.

\item [(ii)] Local Lipschitz continuity: There exists a constant
$C_{\mathcal{N}_0}(T,R,Q_0,u_0)>0$ such that
\begin{align}\label{lip}
\begin{aligned}
&\|\mathcal{N}_0(Q_0)(\uhone,\Qhone)-\mathcal{N}_0(Q_0)(\uhtwo,\Qhtwo)\|_{Y_0} \\
&\,\qquad\qquad
\leq  C_{\mathcal{N}_0}(T,R,Q_0,u_0)\|(\uhone-\uhtwo,\Qhone-\Qhtwo)\|_{X_0}\,.
\end{aligned}
\end{align}

\item [(iii)] Local boundedness:
There exists a constant $\CR(u_0,Q_0)>0$ independent of\, $T$ and $R$ such that
\begin{equation}\label{bounded}
\|\mathcal{N}_0(Q_0)(\uhone,\Qhone)\|_{Y_0}
\leq C_{\mathcal{N}_0}(T,R,Q_0,u_0)\|(\uhone,\Qhone)\|_{X_0}+
\| \mathcal{E}(u_0, Q_0)\|_{Y_0}\,.
\end{equation}
%\textbf{better $\CR = $?}

\item [(iv)]
For $R>0$ fixed we have $\lim_{T\to 0}C_{\mathcal{N}_0}(T,R,Q_0,u_0)= 0$.
\end{itemize}
\end{prop}

\begin{proof}
We divide the proof into several steps.
In order to simplify the presentation,
the dependence of the generic constant $C$ on $\O$, $\xi$, $\nu$
and $(u_0,Q_0)$  will be neglected. Moreover, we will skip the time interval $(0,T)$ and domain $\Omega$ in the vector-valued functions spaces for better readability, e.g. we denote $L^p(L^q)=L^p(0,T;L^q(\Omega))$. For any function $F:\R^k\to \R^\ell$ with $k$, $\ell\in \N$, and any
points $a_1$, $a_2\in \R^k$ we define
\begin{align*}
 \llbracket F(a) \rrbracket = F(a_1) - F(a_2)\,.
\end{align*}
Note that by the definitions of $P_\sigma\colon H^{-1}(\Omega;\R^d)\to H^{-1}_\sigma(\Omega)$ and $\diverg \colon L^2(\Omega;\R^{d\times d})\to H^{-1}(\Omega;\R^d)$ we can estimate
the $H^1(\VNSprime)$ norm of the divergence of the  difference of the fields
in the first component in $\mathcal{N}_0$ in $Y_u$ by their $H^1(L^2)$-norm.
Therefore all estimates in the proof of the Lipschitz continuity involve the
$H^1(L^2)$-norm and will be accomplished based on the fact that most of
the expressions are bilinear or trilinear.

\medskip

\noindent
\textit{Proof of (i): The range of $\mathcal{N}_0$ lies in $Y_0$.}
Fix $(\uh, \Qh)\in X_0$. The compatibility condition for
the initial conditions of elements in $Y_0$ in \eqref{Y0} follows
from $(u_0, Q_0)\in \Zspace$, \eqref{X0} and \eqref{newnon} since for all
$(\uh, \Qh)\in Z_0$,
\begin{align*}
\mathcal{N}_0(Q_0)\begin{pmatrix} \uh\\ \Qh \end{pmatrix}\bigg|_{t=0}
&=\mathcal{N}(Q_0)\begin{pmatrix} \uh+u_0\\ \Qh+Q_0 \end{pmatrix}\bigg|_{t=0}
+ \mathcal{S}(Q_0)\begin{pmatrix} u_0\\ Q_0 \end{pmatrix} \\
&=\mathcal{N}(Q_0)\begin{pmatrix} u_0\\ Q_0 \end{pmatrix}\bigg|_{t=0}
+ \mathcal{S}(Q_0)\begin{pmatrix} u_0\\ Q_0 \end{pmatrix} \\
&=\mathcal{E} \begin{pmatrix} u_0\\ Q_0 \end{pmatrix}
\in L^2_{\sigma}(\O)\times
H^1_{\Gamma_D}(\O)\,.
% \text{ for all }
% \begin{pmatrix} \uh\\ \Qh \end{pmatrix}\in X_0\,.
\end{align*}
Moreover, $\mathcal{N}_0(Q_0)(\uh, \Qh)\in Y_u\times Y_Q$ follows by inspection of all terms in the
definition of $\mathcal{N}_0$ in the same way as in the proof of (ii). %\textbf{check!}

\medskip

\noindent
\textit{Proof of (ii): Local Lipschitz continuity of $\mathcal{N}_0$.}
Let 
\begin{align*}\label{xt}
X_T&= \left\{(u,Q)\in X_u\times X_Q:\mathcal{T}(Q)=
\left( Q_D, Q_N \right),
\left(u,Q \right)|_{t=0}=\left(u_0,Q_0 \right)\right\}\,.
\end{align*}
We define $(u_i,Q_i)=(\uhi+u_0,\Qhi+Q_0)\in X_T$
and $(\overline{u},\overline{Q})=(u_1-u_2,Q_1-Q_2)\in X_0$,
where we identify as usual a function independent of $t$ with its
extension to $(0,T)$ as a constant function. By definition,
\begin{align*}
 \llbracket \mathcal{N}_0(Q_0)(\uh, \Qh) \rrbracket
= \llbracket \mathcal{N}(Q_0)(u, Q) \rrbracket
\end{align*}
and since $\mathcal{N}$
involves only spatial derivatives, we infer
 \begin{equation*}
    \mathcal{N}(Q_0)(u_1,Q_1)\big|_{t=0}-\mathcal{N}(Q_0)(u_2,Q_2)\big|_{t=0}=0\,.
 \end{equation*}
% Hence  \eqref{norm} and \eqref{newnon} implies that
Hence \eqref{lip} is equivalent to
  \begin{equation}\label{lipnew}
    \|\mathcal{N}(Q_0)(u_1,Q_1)-\mathcal{N}(Q_0)(u_2,Q_2)\|_{Y_u\times Y_Q}
    \leq C_{\mathcal{N}_0}(T,R,Q_0,u_0)\|(\bar{u},\bar{Q})\|_{X_0}
    \end{equation}
for all $(u_i,Q_i)\in X_T$ such that
\begin{equation}\label{distance}
\|(u_i,Q_i)-(u_0,Q_0)\|_{X_0}\leq R\,.
\end{equation}
The proof of the Lipschitz continuity
requires
additional estimates and is therefore divided
into several steps in which we estimate the differences between the
various terms in the operators.

\medskip

\noindent
\textit{Step 1: Uniform bounds.}
We have for $i=1$, $2$ uniform bounds in space--time,
\begin{align}\label{f1}
\begin{aligned}
\|Q_i-Q_0\|_{L^{\infty}(\Omega_T)}^2
+ \|\dt{Q}_{i}-\dt{Q}_{0}\|^2_{L^2(L^{\infty})} %+ \|\dt{Q}_{i}\|^2_{L^2(L^{\infty})}
 &\, \leq CT^{\frac 14}R^2\,,\\
\|\bar{Q}\|_{L^{\infty}(\Omega_T)}^2
+\|\dt{\bar{Q}}\|^2_{L^2(L^{\infty})}
&\,  \leq CT^{\frac 14}\|\bar{Q}\|^2_{X_Q} \,,
\end{aligned}
\end{align}
as well as uniform bounds in time  for higher-order norms in space,
\begin{align}\label{f2}
\begin{aligned}
\|   Q_i\|_{L^{\infty}(H^k)}
&\, \leq C\big(R+\|Q_0\|_{H^k}\big),&&\quad k\in \{0,1,2\}\,,\\
\|   u_i\|_{L^{\infty}(H^k)}
&\, \leq C\big(R+\|u_0\|_{H^k}\big),&&\quad k\in \{0,1\}\,, \\
\| \bar{u}\|_{L^{\infty}(H^1)}
+\|\bar{Q}\|_{L^{\infty}(H^2)}
&\, \leq C\barg{ \|\bar{u}\|_{X_u} + \|\bar{Q}\|_{X_Q}}\,. &&
%
% \|\bar{Q}\|_{L^{\infty}(H^k)}
% &\, \leq C\|\bar{Q}\|_{X_Q},\quad \,k\in \{0,1,2\}\,, \\
% \| \bar{u}\|_{L^{\infty}(H^k)}
% &\, \leq C\|\bar{u}\|_{X_u}\,,\quad k\in \{0,1\}\,.
\end{aligned}
\end{align}
Note carefully, that the constants are independent of $T$.
More precisely, by the interpolation result \eqref{interpo1},
 \begin{equation}
 \|Q_i-Q_0\|_{L^{\infty}(\Omega_T)}
\leq C\|Q_i-Q_0\|_{L^{\infty}(H^1)}^{\frac 12}
\|Q_i-Q_0\|_{L^{\infty}(H^2)}^{\frac 12}\,.
\end{equation}
We apply \eqref{interpo2} with $H=H_1=H^k(\O)$, $0\leq k\leq 2$
to $Q_i-Q_0$, observe that the term related to the initial conditions vanishes since
$Q_i|_{t=0}=Q_0$ and obtain for $i=1,\,2$
\begin{equation}\label{f1axx}
  \begin{split}
&\|Q_i-Q_0\|_{L^{\infty}(H^1)}\|Q_i-Q_0\|_{L^{\infty}(H^2)}\\
&\,\quad      \leq  C\|Q_i-Q_0\|_{L^2(H^1)}^{\frac 12}\|Q_i-Q_0\|_{H^1(H^1)}^{\frac 12}\|Q_i-Q_0\|_{L^2(H^2)}^{\frac 12}\|Q_i-Q_0\|^{\frac 12}_{H^1(H^2)}\\
&\,\quad      \leq  CT^{\frac 14}\|Q_i-Q_0\|_{L^{\infty}(H^1)}^{\frac 12}\|Q_i-Q_0\|_{H^1(H^1)}^{\frac 12}\|Q_i-Q_0\|_{H^1(H^2)}\\
&\,\quad      \leq  CT^{\frac 14}\|Q_i-Q_0\|^2_{H^1(H^2)}
\leq  CT^{\frac 14}R^2\,,
\end{split}
\end{equation}
where we used \eqref{distance} in the last step. The uniform bound for $Q_i-Q_0$
in~\bref{f1} follows immediately. To estimate the time derivatives of $Q_i-Q_0$,
we use H\"{o}lder's inequality and \eqref{interpo3} and find
for $i=1,\,2$ that
\begin{equation}\label{qtest}
\begin{split}
       &\|\dtind{Q}{i}-\dtind{Q}{0}\|^2_{L^2(L^{\infty})}\\
        &\leq C\|\dtind{Q}{i}-\dtind{Q}{0}\|_{L^2(L^2)}^{\frac 12}
\|\dtind{Q}{i}-\dtind{Q}{0}\|^{\frac 32}_{L^2(H^2)}\\
        &\leq CT^{\frac 14}\|\dtind{Q}{i}-\dtind{Q}{0}\|_{L^{\infty}(L^2)}^{\frac 12}
\|\dtind{Q}{i}-\dtind{Q}{0}\|^{\frac 32}_{L^2(H^2)}\\
        &\leq CT^{\frac 14}\barg{ \|\dtind{Q}{i}-\dtind{Q}{0}\|_{H^1(L^2)}^{\frac 12}
+\|(\dtind{Q}{i}-\dtind{Q}{0})|_{t=0}
\|_{L^2(\O)}^{\frac 12} }\|\dtind{Q}{i}-\dtind{Q}{0}\|^{\frac 32}_{L^2(H^2)}\\
        &\leq CT^{\frac 14}\|Q_i-Q_0\|^2_{X_Q}\leq CT^{\frac 14}R^2\,.
     \end{split}
    \end{equation}
In the last step, we used the definition \eqref{XQ} of the norm in $X_Q$.
The estimates for $\bar{Q}$ are analogous and the proof of~\bref{f1} is complete.
%     In a similar way, we get
%      \begin{equation}\label{qtest1}
%      \|\bar{Q}_t\|^2_{L^2(L^{\infty})}\leq CT^{\frac 14}\|\bar{Q}\|^2_{X_Q}.
%     \end{equation}
To verify~\bref{f2} we employ the
triangle inequality, \eqref{interpo2} and \eqref{distance} and find
for $i=1,\,2$ and $k=0,\,1,\,2$ that
\begin{equation}\label{qiest}
  \begin{split}
    \|   Q_i\|_{L^{\infty}(H^k)}
    &\leq \|Q_i-Q_0\|_{L^{\infty}(H^k)}+\|Q_0\|_{L^{\infty}(H^k)}\\
    &\leq C\big(\|Q_i-Q_0\|_{H^1(H^k)}
+\| (Q_i-Q_0)|_{t=0}\|_{H^k}\big)+\|Q_0\|_{H^k}\\
    &\leq C\big(R+\|Q_0\|_{H^k}\big)
  \end{split}
\end{equation}
and
\begin{equation}\label{qbarest}
   \|\bar{Q}\|_{L^{\infty}(H^k)}\leq C(\|\bar{Q}\|_{H^1(H^k)}
+\|\bar{Q}|_{t=0}\|_{H^k})=C\|\bar{Q}\|_{X_Q}\,.
\end{equation}
The estimates for $\bar{u}$ and $u_i$ are similar and therefore~\bref{f2}
has been established.
% \begin{equation}\label{uiest}
%   \|   u_i\|_{L^{\infty}(H^k)}\leq C\big(R+\|u_0\|_{H^k}\big)
% ,\quad i=1,2;\quad 0\leq k\leq 1.
% \end{equation}
% \begin{equation}\label{ubarest}
%   \|   \bar{u}\|_{L^{\infty}(H^k)}\leq C\|\bar{u}\|_{X_u},\quad 0\leq k\leq 1.
% \end{equation}

\medskip

\noindent
\textit{Step 2: Estimates for differences of viscosities.}
Note that by the fundamental theorem of calculus,
\begin{align*}%\label{diff1}
    \nu(Q_1)-\nu(Q_2)
&=\int^1_0\frac{d}{d\tau}\nu\barg{ \tau Q_1+(1-\tau)Q_2 }\dv{\tau}
% \\&
=\int^1_0(\nabla\nu)\barg{ \tau Q_1+(1-\tau)Q_2 }:\bar{Q}\dv{\tau}\,,
\end{align*}
and by \eqref{f1}  and for $i=1$, $2$,
% \begin{equation}\label{nuest1}
%   \|\nu(Q_1)-\nu(Q_2)\|_{L^{\infty}(\Omega_T)}
% \leq C(R,\nu, Q_0)\|\bar{Q}\|_{L^{\infty}(\Omega_T)}
% \leq C(R,\nu, Q_0) T^{\frac{1}{8}}\|\bar{Q}\|_{X_Q}\,,
% \end{equation}
% analogously for $i=1$, $2$,
% \begin{equation}
%   \|\nu(Q_i)-\nu(Q_0)\|_{L^{\infty}(\Omega_T)}
% \leq C(R,\nu, Q_0)\|Q_i-Q_0\|_{L^{\infty}(\Omega_T)}\leq C(R,\nu, Q_0) T^{\frac {1}{8}}\,.
% \end{equation}
\begin{equation}\label{nuest1}
\begin{aligned}
  \|\nu(Q_1)-\nu(Q_2)\|_{L^{\infty}(\Omega_T)}
&\, \leq C(R,\nu, Q_0)\|\bar{Q}\|_{L^{\infty}(\Omega_T)}
\leq C(R,\nu, Q_0) T^{\frac{1}{8}}\|\bar{Q}\|_{X_Q}\,,
\\
\|\nu(Q_i)-\nu(Q_0)\|_{L^{\infty}(\Omega_T)}
&\, \leq C(R,\nu, Q_0)\|Q_i-Q_0\|_{L^{\infty}(\Omega_T)}
\leq C(R,\nu, Q_0) T^{\frac {1}{8}}\,.
\end{aligned}
\end{equation}
If one differentiates the integral representation, then one finds
\begin{align*}
 &\, \p_t\barg{ \nu(Q_1)-\nu(Q_2) } \\
&\, \quad = \int_0^1\bset{
(\nabla^2 \nu)\barg{ \tau Q_1+(1-\tau)Q_2 }[\bar{Q}, \tau \dtind{Q}{1}+(1-\tau)\dtind{Q}{2}]
+(\nabla\nu)\barg{ \tau Q_1+(1-\tau)Q_2 } : \bar{Q}_t
}\dv{\tau}\,.
\end{align*}
We deduce from~\bref{viscosity}, \eqref{f1}, \eqref{f2} and the foregoing formula that
for $i=1$, $2$ and a.e.\ in $\ts{T}$,
\begin{align}\label{vis3}
\begin{aligned}
  |\p_t(\nu(Q_1))-\p_t(\nu(Q_2)))|
&\, \leq C(R) \barg{ |{\bar{Q}}| +  |\dt{\bar{Q}}|  }\,, \\
  |\p_t(\nu(Q_i))-\p_t(\nu(Q_0))|
&\, \leq C(R)\barg{ |Q_{i}-Q_{0}| + |\dtind{Q}{i}-\dtind{Q}{0}| }\,.
\end{aligned}
\end{align}
Finally note that
\begin{align*}
 \| \dtind{Q}{i} \|_{L^\infty(\ts{T})} \leq C \| Q_i\|_{X_0}\,.
\end{align*}

%
% We start with the term
% \begin{equation}\label{nonline1}
%   P_{\sigma}\div\big((\nu(Q)-\nu(Q_0))\Lsym{u}\big).
% \end{equation}
%
%

\medskip

\noindent
\textit{Step 3: Estimates for differences of viscous stress tensor.}
We verify  the estimate
\begin{equation}\label{vis9}
  \begin{split}
    &\bnorm{ \bigl\llbracket
P_{\sigma}\div\barg{ (\nu(Q)-\nu(Q_0)) \Lsym{u} } \bigr\rrbracket }_{H^1(\VNSprime)}
\leq C(\nu,R)T^{\frac {1}{8}}\|(\overline{u},\overline{Q})\|_{X_u\times X_Q}\,.
  \end{split}
\end{equation}
To this end we rewrite this expression as
\begin{equation*}%\label{vis4}
  \begin{split}
    &\left\| P_{\sigma} \div\bsqb{ (\nu(Q_1)-\nu(Q_0))\Lsym{u_1} -
% }-P_{\sigma}\div\barg{
(\nu(Q_2)-\nu(Q_0)) \Lsym{u_2} } \right\|_{H^1(\VNSprime)}\\
    &{ }\leq \left\|\div \bsqb{ (\nu(Q_1)-\nu(Q_2))\Lsym{u_1} +
% }-P_{\sigma}\div\barg{
(\nu(Q_2)-\nu(Q_0))\Lsym{\bar{u}} }\right\|_{H^1(H^{-1})}\\
    & { }\leq \|(\nu(Q_1)-\nu(Q_2))\Lsym{u_1}\|_{H^1(L^2)}
+\|(\nu(Q_2)-\nu(Q_0))\Lsym{\bar{u}}\|_{H^1(L^2)}\,,
  \end{split}
\end{equation*}
Expressing these norms as $L^2(L^2)$-norms of the functions and their first order derivative in time leads to 
four higher-order and two lower-order terms.
For the higher-order terms we find by \eqref{vis3} and \eqref{f1}
\begin{align*}%\label{vis5}
   \|\p_t(\nu(Q_1)-\nu(Q_2))\Lsym{u_1}\|_{L^2(\Omega_T)}
    \leq & C(R) \norm{ (|\bar{Q}| + |\dt{\bar{Q}}|) \, |\Lsym{u_1}| }_{L^2(\Omega_T)}\\
    \leq & C(R)\|{\bar{Q}}\|_{H^1(L^{\infty})}\|\Lsym{u_1}\|_{L^{\infty}(L^2)}
% \\    \leq &
\leq C(R)T^{\frac {1}{8}}\|\bar{Q}\|_{X_0}\,,
\end{align*}
and analogously
\begin{equation*}%\label{vis6}
  \begin{split}
    \|\p_t(\nu(Q_i)-\nu(Q_0))D\bar{u}\|_{L^2(\Omega_T)}
    \leq & C(R) \|(|Q_i-Q_0|+|\p_t(Q_i-Q_0)|\,|D\bar{u}|\|_{L^2(\Omega_T)}\\
    \leq &C(R)\|Q_i-Q_0\|_{H^1(L^{\infty})}\|D\bar{u}\|_{L^{\infty}(L^2)}\\
     \leq &C(R)T^{\frac {1}{8}}\|\bar{u}\|_{X_u}.
  \end{split}
\end{equation*}
We obtain for the remaining two higher-order terms by \eqref{f1}, \eqref{f2} and \eqref{nuest1}
\begin{equation*}%\label{vis7}
  \begin{split}
    \|(\nu(Q_1)-\nu(Q_2))\dt{\Lsym{u_{1}}}\|_{L^2(\Omega_T)}
    \leq & C(R)\|\bar{Q}\|_{L^{\infty}(\Omega_T)}\|\dt{\Lsym{u_{1}}}\|_{L^2(\Omega_T)}\\
     \leq & C(R)\|\bar{Q}\|_{L^{\infty}(\Omega_T)}\barg{ \|u_1-u_0\|_{X_u}
+\|\dt{\Lsym{u_{0}}}\|_{L^2(\Omega_T)} }\\
    \leq & C(R)T^{\frac {1}{8}}\|\bar{Q}\|_{X_Q}
  \end{split}
\end{equation*}
and
\begin{align*}%\label{vis8}
\|(\nu(Q_2)-\nu(Q_0))\dt{\Lsym{\bar{u}}}\|_{L^2(\Omega_T)}
\leq & C(R)T^{\frac {1}{8}}\|\dt{\Lsym{\bar{u}}}\|_{L^2(\Omega_T)}
% \\
\leq C(R)T^{\frac {1}{8}}\|\bar{u}\|_{X_u}\,.
\end{align*}
To estimate the lower-order terms in the $H^1(L^2)$-norm we
use \eqref{nuest1} and obtain
\begin{align*}%\label{lowest1}
&\|(\nu(Q_1)-\nu(Q_2))\Lsym{u_1}\|_{L^2(L^2)}
+\|(\nu(Q_2)-\nu(Q_0))\Lsym{\bar{u}}\|_{L^2(L^2)}\\
\leq & C(R)T^{\frac {1}{8}}
\barg{ \|\bar{Q}\|_{X_Q}\|\Lsym{u_1}\|_{L^2(L^2)}
+\|Q_2-Q_0\|_{X_Q}\|\Lsym{\bar{u}}\|_{L^2(L^2)} }
% \\
\leq  C(R)T^{\frac {1}{8}}\|(\bar{u},\bar{Q})\|_{X_u\times X_Q}\,.
\end{align*}
The combination of the foregoing estimates implies the assertion of this step.

\medskip

\noindent
\textit{Step 4: Fundamental estimates for bilinear forms.}
Suppose that $P_1$, $P_2$ are
time dependent tensor fields with initial value $P_0$, that $\bar P = P_1-P_2$, and that
$B:\M^{d\times d}\times \M^{d\times d} \to \M^{d\times d}$
is a bilinear form with constant coefficients. Then
\begin{align*}
 & \norm{ \llbracket B(Q-Q_0, P) \rrbracket }_{H^1(L^2)} \\
&\,\quad \leq
C T^{\frac{1}{8}}  R\barg{
\norm{ \bar P }_{L^\infty(L^2)} + \norm{ \bar P}_{H^1(L^2)} )}
% + C(R, \|P_0\|)
+ CT^{\frac{1}{8}} \norm{\bar Q}_{X_Q}
\barg{ \norm{ P_2}_{L^\infty(L^2)} + \norm{ {P_2}}_{H^1(L^2)} }
\end{align*}
where we assume that all norms on the right-hand side are finite.
In particular,% \textbf{check constants}
\begin{align*}
 \norm{ \llbracket B(Q-Q_0, P) \rrbracket }_{H^1(L^2)}
\leq C(R) T^{\frac{1}{8}}\,\norm{ (\bar u, \bar Q) }_{X_0}
\end{align*}
for $P\in \bset{ \nabla u, \Lsym{u}, \Lskew{u}, \Delta Q }$.
%\begin{align*}
% \norm{ \llbracket B(Q-Q_0, P) \rrbracket }_{H^1(L^2)}
%\leq C(R) T^{\frac{1}{8}}\,\norm{ (\bar u, \bar Q) }_{X_0}\,.
%\end{align*}
In fact, by the triangle inequality and the product rule for bilinear forms,
\begin{align*}
 &\, \norm{ B(Q_1-Q_0, P_1) - B(Q_2-Q_0, P_2)}_{H^1(L^2)}\\
&\, \qquad \leq \norm{ B(Q_1-Q_0, \bar P)}_{H^1(L^2)}
+ \norm{ B(\bar Q, P_2)}_{H^1(L^2)}\\
&\,\qquad  \leq
\norm{ B(\dtind{Q}{1}-\dtind{Q}{0}, \bar P)}_{L^2(L^2)}
+\norm{ B(Q_1-Q_0, \dt{\bar P})}_{L^2(L^2)}
+\norm{ B(Q_1-Q_0, \bar P)}_{L^2(L^2)}
\\
&\, \qquad \quad
+\norm{ B(\dt{\bar Q}, P_2)}_{L^2(L^2)}
+\norm{ B(\bar Q, \dtind{P}{2})}_{L^2(L^2)}
+\norm{ B(\bar Q, P_2)}_{L^2(L^2)}\\
&\, \qquad  \leq C
\bset{
\norm{ \dtind{Q}{1}-\dtind{Q}{0}}_{L^2(L^\infty)}
\norm{ \bar P }_{L^\infty(L^2)}
+\norm{ Q_1-Q_0}_{L^\infty(\ts{T})} \norm{ \bar P}_{H^1(L^2)}
%+\norm{ Q_1-Q_0}_{L^\infty(\ts{T}} \norm{ \bar P}_{L^2(\ts{T})}
\\ &\, \qquad  \quad
+\norm{ \dt{\bar Q} }_{L^2(L^\infty)} \norm{ P_2}_{L^\infty(L^2)}
+\norm{ \bar Q}_{L^\infty(\ts{T})} \norm{ {P_2}}_{H^1(L^2)} }\,.
\end{align*}
The assertion follows from \bref{f1} and \bref{f2}.

\medskip

\noindent 
\textit{Step 5: Fundamental estimates for trilinear forms.}
Suppose that $P_1$, $P_2$ are
time dependent tensor fields with initial values $P_0$ and that
$E:\M^{d\times d}\times \M^{d\times d}\times \M^{d\times d}
\to \M^{d\times d}$
is a trilinear form with constant coefficients. Then %\textbf{check $\|P_0\|$}
\begin{align*}
 & \norm{ \llbracket E(Q,Q, P) - E(Q_0,Q_0, P) \rrbracket }_{H^1(L^2)} \\
&\,\quad \leq
C T^{\frac{1}{8}}  R\barg{
\norm{ \bar P }_{L^\infty(L^2)} + \norm{ \bar P}_{H^1(L^2)} )}
+ C(R, \|P_0\|_{L^2}) T^{\frac{1}{8}} \norm{\bar Q}_{X_Q}
\barg{ \norm{ {P_2}}_{H^1(L^2)} +\norm{ P_0}_{L^2}  }
\end{align*}
where we assume that all norms on the right-hand side are finite.
In particular, 
\begin{align*}
 \norm{ \llbracket E(Q,Q, P) - E(Q_0,Q_0, P) \rrbracket }_{H^1(L^2)}
\leq C(R) T^{\frac{1}{8}}\,\norm{ (\bar u, \bar Q)}_{X_0}
\end{align*}
for $P\in \bset{ \nabla u, \Lsym{u}, \Lskew{u},\Delta Q }$.
%\begin{align*}
% \norm{ \llbracket E(Q,Q, P) - E(Q_0,Q_0, P) \rrbracket }_{H^1(L^2)}
%\leq C(R) T^{\frac{1}{8}}\,\norm{ (\bar u, \bar Q)}_{X_0} \,.
%\end{align*}
To see this, note that
\begin{align*}
&\, E(Q_1,Q_1, P_1) - E(Q_0,Q_0, P_1) - E(Q_2,Q_2, P_2) + E(Q_0,Q_0, P_2)  \\
 &\,\qquad =  E(Q_1,Q_1, \bar P) -   E(Q_0,Q_0, \bar P)
+E(Q_1,Q_1, P_2)    - E(Q_2,Q_2, P_2) \\
 &\,\qquad =  E(Q_1-Q_0,Q_1, \bar P) +   E(Q_0,Q_1-Q_0, \bar P)
+E(Q_1-Q_2,Q_1, P_2)    + E(Q_2,Q_1-Q_2, P_2)\,.
\end{align*}
We need to estimate this sum in the $H^1(L^2)$-norm. By H\"older's inequality and the
product rule for trilinear forms, each term leads to four terms that need to be
estimated in $L^2(L^2)$. For the first term we find
\begin{align*}
 &\, \norm{ E(Q_1-Q_0,Q_1, \bar P) }_{H^1(L^2)}\\
&\,\quad \leq \norm{ E(\dtind{Q}{1}-\dtind{Q}{0},Q_1, \bar P)
+ E(Q_1-Q_0,\dtind{Q}{1}, \bar P)
+ E(Q_1-Q_0,Q_1, \dt{\bar P})  }_{L^2(L^2)}
\\ &\,\quad \qquad
+\norm{ E(Q_1-Q_0,Q_1, \bar P) }_{L^2(L^2)} \\
&\,\quad \leq C\bset{
\norm{\dtind{Q}{1}-\dtind{Q}{0}}_{L^2(L^\infty)} \norm{Q_1}_{L^\infty(\ts{T})}
\norm{\bar P}_{L^\infty(L^2)}
+\norm{Q_1-Q_0}_{L^\infty(\ts{T})} \norm{\dtind{Q}{1}}_{L^2(L^\infty)}
\norm{\bar P}_{L^\infty(L^2)}\\
&\,\qquad + \norm{Q_1-Q_0}_{L^\infty(\ts{T})} \norm{Q_1}_{L^\infty(\ts{T})}
\norm{\bar P}_{H^1(L^2)}  }\,.
\end{align*}
The second term can be estimated the same way. For the third  term one finds
\begin{align*}
 &\, \norm{ E(Q_1-Q_2,Q_1, P_2) }_{H^1(L^2)}\\
&\,\quad \leq \norm{ E(\dtind{Q}{1}-\dtind{Q}{2},Q_1, P_2)
+ E(Q_1-Q_2,\dtind{Q}{1}, P_2)
+ E(Q_1-Q_2,Q_1, \dtind{P}{2})}_{L^2(L^2)}
\\ &\,\quad \qquad
  +\norm{ E(Q_1-Q_2,Q_1, P_2) }_{L^2(L^2)} \\
&\,\quad \leq C \bset{
\norm{\dtind{Q}{1}-\dtind{Q}{2}}_{L^2(L^\infty)} \norm{Q_1}_{L^\infty(\ts{T})}
\norm{P_2}_{L^\infty(L^2)}
+\norm{Q_1-Q_2}_{L^\infty(\ts{T})} \norm{\dtind{Q}{1}}_{L^2(L^\infty)}
\norm{P_2}_{L^\infty(L^2)}\\
&\,\qquad + \norm{Q_1-Q_2}_{L^\infty(\ts{T})} \norm{{Q}_{1}}_{L^\infty(\ts{T})}
\norm{P_2}_{H^1(L^2)}  }.
\end{align*}
The fourth term can be estimated as before.

\medskip

\noindent
\textit{Step 6: Estimates for additional stress tensors in the fluid equation.}
 We have
\begin{align*}%\label{longestimate}
  \bnorm{ \llbracket P_\sigma \diverg \bset{\sigma(Q-Q_0,\Delta Q)
+\xi\barg{ \tau_2(Q,\Delta Q)-\tau_2(Q_0,\Delta Q) } } \rrbracket
%\big|_{Q_2}^{Q^1}
}_{H^1(\VNSprime)}
\leq C(R)T^{\frac{1}{8}}\|\bar{Q}\|_{X_Q}\,.
\end{align*}
This follows for $\sigma(Q-Q_0,\Delta Q)$ and the bilinear part in
$\tau_2(Q,\Delta Q)-\tau_2(Q_0,\Delta Q)$ from Step~4 and for the trilinear part
$ Q\tr(Q\Delta Q)-Q_0\tr(Q_0\Delta Q)$ from Step~5.

\medskip

\noindent
\textit{Step 7: The coupling term in the evolution of the tensor field.} We have
\begin{align*}
\bnorm{ \big\llbracket S_1(\nabla u,Q)+\xi S_2(\nabla u,Q)
\bigr\rrbracket }_{H^1(\VNSprime)}
\leq   C(R)T^{\frac{1}{8}}\|(\bar{u},\bar{Q})\|_{X_u\times X_Q}\,.
\end{align*}
This follows for the bilinear part in
$S_1(\nabla u,Q)+\xi S_2(\nabla u,Q)$ from Step~4 and for the trilinear part
$ Q\tr(Q\nabla u)-Q_0\tr(Q_0\nabla u)$ from Step~5.

\medskip

\noindent
\textit{Step 8: Additional lower-order terms.}
The terms $u\otimes u$, $(u\cdot\nabla) Q+L(Q)$ and
\begin{align*}
 \mathcal{J}(Q)
=P_{\sigma}\div\bsqb{ \tau_1(Q)+\sigma(Q,L(Q))+\xi\tau_2(Q,L(Q))
-\frac{2\xi}{d}L(Q)}
\end{align*}
are of lower-order and lead to the estimates
\begin{equation*}\label{lowerorderestim}
 \begin{split}
   \|\mathcal{J}(Q_1)-\mathcal{J}(Q_2)\|_{Y_u}
  &\, \leq T^{\frac{1}{8}}C(R)\|\bar{Q}\|_{X_Q} \,, \\
\|\div \barg{ u_1\otimes u_1-  u_2\otimes u_2 }\|_{Y_u}
&\, \leq T^{\frac{1}{8}} C(R)\|\bar{u}\|_{X_u}\,, \\
\|(u_1\cdot \nabla) Q_1 +L(Q_1)-(u_2\cdot \nabla) Q_2 -L(Q_2)\|_{Y_Q}
&\, \leq T^{\frac{1}{8}} C(R)\|(\bar{u},\bar{Q})\|_{X_u\times X_Q}\,.
 \end{split}
\end{equation*}
These estimates can be done the same way as in Step 4 and 5.
%\textbf{check that nobody was left behind.}

\medskip

\noindent
\textit{Proof of (iii): Boundedness of $\mathcal{N}_0$.}

% In the last step, we use the compatibility condition in $\Zspace$
% and $\p_t Q_D|_{\Gamma_D}\equiv 0$.
If suffices to show
that \eqref{bounded} is a consequence of \eqref{lip}.
In fact, the choice of $(u_2,Q_2)=0$ in \eqref{lip} implies
\begin{equation}
\|\mathcal{N}_0(Q_0)(u_1,Q_1)-\mathcal{N}_0(Q_0)(0,0)\|_{Y_0}
\leq C_{\mathcal{N}_0}(T,R)\|(u_1,Q_1)\|_{Y_0}
\end{equation}
and the assertion follows by  the triangle inequality
since $\mathcal{N}_0(Q_0)(0,0) = \mathcal{E}(u_0, Q_0)$, see the proof of (i).

\medskip

\noindent
\textit{Proof of (iv): Asymptotic behaviour of the constant.} This assertion
follows from the scaling of the constants in step (ii) in $T$.

\end{proof}

\subsection{Proof of Theorem~\ref{localstrong}}\label{finallyproof}

By \eqref{trans} and \eqref{newnon}, the proof of Theorem \ref{localstrong}
can be reduced to the statement that the nonlinear mapping
\begin{equation}
  \mathscr{L}(Q_0):=\mathcal{L}^{-1}(Q_0)\mathcal{N}_0(Q_0):X_0\to X_0
\end{equation}
has a unique fixed-point. By~\bref{stronginverse} and \eqref{lip} we find for all
$(\uhi,\Qhi)\in \BXzeroR$ that
\begin{align*}
&\bnorm{ \mathcal{L}^{-1}(Q_0)\mathcal{N}_0(Q_0) (\uhone,\Qhone)
- \mathcal{L}^{-1}(Q_0)\mathcal{N}_0(Q_0)(\uhtwo,\Qhtwo) }_{X_0}
\\ &\, \qquad
  \leq  C_{\mathcal{L}}\|\mathcal{N}_0(Q_0)
  (\uhone,\Qhone)-\mathcal{N}_0(Q_0)(\uhtwo,\Qhtwo)\|_{Y_0}
\\ &\, \qquad
  \leq  C_{\mathcal{L}}C_{\mathcal{N}_0}(T,R,Q_0)
  \|(\uhone-\uhtwo,\Qhone-\Qhtwo)\|_{X_0}\,.
\end{align*}
Therefore $\mathscr{L}(Q_0)$ is a contraction mapping for $T\ll 1$.
A similar argument shows  that $\mathscr{L}$ maps $\BXzeroR$ into itself. In fact, by
by \eqref{bounded}
\begin{align*}
\bnorm{ \mathscr{L}(Q_0)(\uhone, \Qhone) }_{X_0}
&\leq C_{\mathcal{L}} \bnorm{ \mathcal{N}_0(Q_0)(\uhone,\Qhone) }_{Y_0}\\
&\leq C_{\mathcal{L}}\barg{ C_{\mathcal{N}}(T,R,Q_0)\|(\uhone,\Qhone)\|_{X_0}+
\|\mathcal{E}(u_0,Q_0)\|_{Y_0}}
\end{align*}
and this estimate allows us to fix $R \gg 1$ large enough
and $T \ll 1$ small enough in such a way that
% to ensure that
% $C_{\mathcal{L}}\CR(u_0,Q_0)\leq \frac R2$. Then
\begin{align*}
      \bnorm{ \mathscr{L}(Q_0)(\uhone,\Qhone) }_{X_0}
&\leq C_{\mathcal{L}}
C_{\mathcal{N}_0}(T,R)\|(\uhone,\Qhone)\|_{X_0} + \frac R2 \leq R\,.
    \end{align*}
We conclude from Banach's fixed-point theorem that $\mathscr{L}$ possess a
unique fixed-point $(\uh,\Qh)\in X_0$ and this fixed-point is
a  solution of the system \eqref{strongformeqn} subject to \eqref{initialconditions}
and \eqref{mixedboundaryconditions}.

The argument implies the uniqueness as well. Suppose that
there was another solution $(\hat{u}_h,\hat{Q}_h)$ in
$\BXzeroRone$ with $R_1>R$. Choose  $\hat{T}\leq T$ and repeat the above argument
to show the uniqueness of fixed-points of $\mathscr{L}$, which implies $(\uh,\Qh)=(\hat{u}_h,\hat{Q}_h)$
on $(0,\hat{T})\times\O$. Then the uniqueness follows by the continuity argument.\qed

%\bibliographystyle{acm}
%\bibliography{cfs}

\end{document}